%% file: imm.tex
\documentclass[a4paper]{amsart}
\usepackage{amsmath,amscd,amssymb}
\numberwithin{equation}{section}
\usepackage[dvips]{graphicx, psfrag}
\usepackage{bm}
\usepackage[varg]{txfonts}
\usepackage[all]{xy}
\usepackage{amsthm, color, enumerate, latexsym}
\makeatletter\@ifundefined{coloneqq}{}\makeatother
\newcommand{\ndeg}{\mathit{{D}}}
\newcommand{\hdef}{\mathit{{H}}}
\newcommand{\Z}{\mathbb{Z}}
\newcommand{\R}{\mathbb{R}}
\newcommand{\C}{\mathbb{C}}

\newcommand{\ind}{\mathop{\mathrm{ind}}\nolimits}
\newcommand{\rank}{\mathop{\mathrm{rk}}\nolimits}

\newcommand{\imm}{\mathop{\mathrm{Cob}}\nolimits}
\newcommand{\Imm}{\mathop{\mathrm{Imm}}\nolimits}

\newcommand{\SO}{\mathop{\mathit{SO}}\nolimits}
\newcommand{\co}{\mathpunct{\colon}}
\newcommand{\dic}{{\rm Dic}}

\newcommand{\pa}{\partial}
\newcommand{\na}{\natural}
\newcommand{\lo}{\looparrowright}

\newcommand{\abs}[1]{\left\lvert #1\right\rvert}
\newtheorem{thm}{Theorem}[section]

\newtheorem{lem}[thm]{Lemma}

\theoremstyle{definition}
\newtheorem{defn}[thm]{Definition}
\newtheorem{rem}[thm]{Remark}
\newtheorem{ex}[thm]{Example}
\date{\today}
\title[Immersions of $3$-sphere into $4$-space associated with Dynkin diagrams of types $A$ and $D$]
{Immersions of $3$-sphere into $4$-space associated with Dynkin diagrams of types $A$ and $D$}
\author{Shumi Kinjo}
\address{Department of Mathematical Sciences, Faculty of Science, Shinshu University, Matsumoto, 390-8621 Japan.}
\email{skinjo@math.shinshu-u.ac.jp}

\begin{document}\sloppy
\maketitle

\begin{abstract}
We construct two infinite sequences of immersions of the $3$-sphere into $4$-space, parameterized by the Dynkin diagrams of types $A$ and $D$. The construction is based on immersions of $4$-manifolds obtained as the plumbed immersions along the weighted Dynkin diagrams. We compute their Smale invariants and bordism classes of immersions by using Ekholm-Takase's formula in terms of singular Seifert surfaces. In order to construct singular Seifert surfaces for the immersions, we use the Kirby calculus.
\end{abstract}

\section{Introduction}\label{sec:intr}
Composing the universal covering $S^3\to M^3$ with an immersion $M^3 \looparrowright \R^4$ of a $3$-manifold $M$ is a well known method to construct a nontrivial immersion of $S^3$ into $\R^4$ with desired properties. For instance, Milnor has studied in \cite{milnor} immersions of the real projective space $\R P^3$ into $\R^4$ and constructed an immersion $f\circ p \co S^3 \to \R^4$ with even normal degree by composing the universal double covering $p\co S^3 \to \R P^3$ and an immersion $f\co \R P^3 \looparrowright \R^4$. 
In \cite{melikhov}, Melikhov has constructed a similar immersion $S^3 \to \R^4$ to the one as above with non-trivial stable Hopf invariant. More precisely, he has defined an immersion as the composite of the universal $8$-fold covering $q\co S^3\to Q^3$ and an embedding $e\co Q^3 \to \R^4$ of the quaternion space $Q^3=S^3/\{\pm 1,\pm i,\pm j,\pm k\}$.
Ekholm and Takase have constructed immersions of the lens space $L(2n,1)$ into $\R^4$ and computed the Smale invariants of the immersions $S^3 \to L(2n,1)\to \R^4$ by using a formula for Smale invariant in terms of singularities of a singular Seifert surface (\cite{e-t}).
In this paper we define immersions of lens spaces $L(n,1)$ and $S^3/\dic_n$ into $\R^4$, and determine the Smale invariants of the immersions $S^3\to L(n,1)\to \R^4$ and $S^3 \to S^3/\dic_n \to \R^4$, where ${\rm Dic}_n$ is a finite subgroup of $S^3$ called the \textit{dicyclic group}. Note that the dicyclic group $\dic_{n}$ has the presentation
\[
\dic_n=\bigl\langle a,x \mid a^n=x^2,xax^{-1}=a^{-1}\bigr\rangle,
\]
and has order $4n$.

$L(n,1)$ and $S^3/\dic_n$ appear naturally in the boundaries of $4$-manifolds obtained by plumbing. A {\em plumbed $4$-manifold} is a manifold obtained by gluing two or more $D^2$-bundles over $S^2$ along their restricted bundles $D^2\times D^2$ with their bases and fibers interchanged (Definition~\ref{def:plum}).
We denote $P(T,m_1,\dots,m_n)$ the manifold obtained by plumbing the $D^2$-bundles over $S^2$ of Euler class $m_i$ along the weighted Dynkin diagram $(T,m_1,\dots,m_n)$. It is known that the boundary of the plumbing along the Dynkin diagrams of types $A$ and $D$ are
\[
\pa P(A_{n-1},2,\cdots,2)=L(n,1),
\]
\[
\pa P(D_{n+2},2,\cdots,2)=S^3/{\rm Dic}_n
\]
(see \cite[\S6.2]{orlik}, \cite[\S8]{h-n-k}).

We recall the construction of some immersions from \cite{e-t}. Considering an immersion of $S^2$ into $\R^4$ with (algebraically) $k$ double points and taking its normal disk bundle, we obtain an immersion of the $2$-disk bundle over $S^2$ of even Euler class $m=2k$ into $\R^4$ (see \cite{e-t}). Plumbing copies of these immersions, we obtain an immersion of $P(T,2k_i)$ into $\R^4$. If we take $T=A_{n-1}$ or $D_{n+2}$ and $m_i=2$, we obtain immersions $\kappa_{{A},n} \co L(n,1) \looparrowright \R^4$ and $\kappa_{{D,n}} \co S^3/\dic_n \looparrowright \R^4$. The composites of the universal coverings with $\kappa_{{A},n}$ and $\kappa_{{D},n}$ are immersions of $S^3$:
\[
f_n \co S^3 \twoheadrightarrow L(n,1) \looparrowright \R^4,
\]
\[
g_n \co S^3 \twoheadrightarrow S^3/\dic_n \looparrowright \R^4.
\]

Let $\Imm[S^m,\R^N]$ denote the group of regular homotopy classes of immersions of $S^m$ into $\R^N$, where the group operation is induced by connected sum of immersions. The Smale-Hirsch h-principle implies that the group $\Imm[S^m,\R^N]$ is isomorphic to the $m{\rm th}$ homotopy group $\pi_m(V_{N,m})$ of the Stiefel manifold $V_{N,m}$ of $m$ frames in $\R^{N}$. The isomorphism is given as taking the differential at each point of $S^m$ and is called the Smale invariant.
In particular, in the case studied in this paper:
\[
\Imm[S^{3},\R^{4}]\cong \pi_3(V_{4,3})\cong \pi_3(\SO_4)\cong \Z\oplus\Z,
\]
the last isomorphism is due to the fact that $\SO_4$ is the trivial $\SO_3$-bundle over $S^3$ (see \cite[\S22]{steenrod}).

Let $\imm(m,N)$ denote the group of bordism classes of immersions of closed oriented $m$-manifolds into $\R^{N}$, where the group operation is induced by disjoint union. The bordism group $\imm(3,4)$ is isomorphic to the stable $3$-stem via  Pontryagin-Thom construction
\[
\imm(3,4)\cong \pi^{S}_3\cong \Z_{24},
\]
see \cite{freedman,wells}. In order to compute the bordism class represented by an immersion $f\co S^3 \to \R^4$ we use the following maps:
\[
\begin{CD}
\pi_3(\SO_4) @>{\iota_*}>> \pi_3(\SO_5)
@>{J}>>
\pi_8(S^5)\cong \pi_3^S,
\end{CD}
\]
where the map $\iota_* \co \pi_3(\SO_4)\cong \Z\oplus\Z \to \Z \cong \pi_3(\SO_5)$ is induced by the inclusion $\iota \co \SO_4 \to \SO_5$. It is known that $\iota_*$ is given by $(a,b)\mapsto a+2b$ (see \cite{steenrod}) and the $J$-homomorphism $J$ is surjective (see \cite[p.~180]{hughes}).
\begin{thm}\label{thm:main}
The Smale invariants of the immersions $f_{n}\colon S^{3}\to L(n,1)\to \R^{4}$ and $g_{n}\colon S^3\to S^3/{\rm Dic}_n \to\R^4$, $n>0$, are given by
\[
\Omega(f_n)=(n^2-1,0)\in\Z\oplus\Z,
\]
\[
\Omega(g_n)=(4n^2+12n-1,0)\in\Z\oplus\Z.
\]
It follows in particular that $f_n$ and $g_n$ represent the elements
\[
[f_{n}]=(n^2-1)\operatorname{mod}{24}\,\in\Z_{24}\approx\pi_{3}^{S},
\]
\[
[g_{n}]=(4n^2+12n-1)\operatorname{mod}{24}\,\in\Z_{24}\approx\pi_{3}^{S}.
\]
Therefore, $f_n$ (resp.~$g_n$) generates the stable $3$-stem if and only if $n$ is a multiple of $6$ (resp.~$3$).
\end{thm}

The paper is organized as follows. In Section~\ref{sect:sss}, we recall the formula for the Smale invariant in terms of singularities which is given by Ekholm-Takase \cite{e-t}. In Section~\ref{sect:kirby}, we summarize some basic facts about Kirby calculus, which we use to construct singular Seifert surfaces for $f_n$ and $g_n$. In Section~\ref{sect:proof}, we prove Theorem~\ref{thm:main}.

\section{Singular Seifert surfaces and Smale invariants in terms of singularities}\label{sect:sss}
\subsection{Smale invariants in terms of singularities}

Hughes and Melvin \cite{hughes-melvin} have shown that the Smale invariant of an embedding $f \co S^{4k-1} \hookrightarrow \R^{4k+1}$ is determined by the signature of a Seifert surface for $f$. Ekholm and Sz\H{u}cs \cite{e-s0} have generalized the result of \cite{hughes-melvin} to the case of immersions by introducing the notion of {\em singular Seifert surfaces}.

In \cite{e-t}, Ekholm-Takase has given a formula for the Smale invariant of an immersion $f\co S^3 \looparrowright \R^4$ by using a singular Seifert surface for $f$.
\begin{defn}\label{defn:sss}
Let $f\co S^3\to\R^4$ be an immersion.
A \textit{singular Seifert surface} for $f$ is a smooth map
$F\co V\to\R^4$
from a compact oriented $4$-manifold $V$ with $\partial V=S^{3}$ which satisfies the following conditions; 
\begin{enumerate}[(i)]
\item The restriction $F|_{\partial V}$ to the boundary equals $f$.
\item $F$ has no singularity near the boundary $\pa V$.
\item For any $p \in V$, the rank $\rank (dF_p)$ of the differential
\[
dF_p \co T_pV \to T_{F(p)}\R^4
\]
satisfies $\rank (dF_p) \ge 2$, and points $q$ with $\rank (dF_q)=2$ are isolated.
\end{enumerate}
\end{defn}

We orient the boundary of an oriented manifold following the ``outward normal first'' convention.
An immersion $f\co S^3\to\R^4$ comes equipped with a stable framing via the orientation preserving bundle isomorphism $\epsilon^1\oplus TS^3\cong f^*T\R^4$, where $\epsilon^1$ denotes the trivial line bundle. Then let $n_{f}\co S^3 \to S^3$ denote the normal vector field of $f$ which gives the positive framing of $\epsilon^1$. The \textit{normal degree} $\ndeg(f)$ is defined by:
\[
\ndeg(f)=\deg n_{f}.
\]

\begin{thm}
[Hughes \cite{hughes}, Kirby-Melvin \cite{k-m}, Takase \cite{takase}, Ekholm-Takase \cite{e-t}]\label{thm:e-t}
Let $f \co S^3 \looparrowright \R^4$ be an immersion and let $F\co V \to \R^4$ be a singular Seifert surface for $f$.
\begin{enumerate}[(a)]
\item \cite{k-m,takase}
The Hirzebruch defect $\hdef(f)$ (of the stable framing induced by $f$) satisfies
\begin{equation}\label{eq:hirzebruch}
\hdef(f)=-3\sigma(V)-\sharp\Sigma^2(F),
\end{equation}
where $\sigma(V)$ is the signature of $V$ and $\sharp\Sigma^2(F)$ is the algebraic number of rank $2$ points of $F$ (see \S 2.2). In fact, the right hand side of \eqref{eq:hirzebruch} depends only on $f$, and does not depend on the choice of singular Seifert surface.

\item \cite{hughes, e-t}
The Smale invariant of $f$ is given by
\begin{equation}\label{eq:smale}
\Omega(f)=
\left(\ndeg(f)-1,\frac{-\hdef(f)-2(\ndeg(f)-1)}{4}\right)\in\Z\oplus\Z,
\end{equation}
where $\ndeg(f)$ is the normal degree of $f$.
\end{enumerate}
\end{thm}

\begin{rem}\label{rem:hughes}
\begin{enumerate}[(a)]
\item In \cite{hughes}, Hughes has given the formula of the Smale invariant of immersion $f\co S^3 \looparrowright \R^4$ which is defined as the restriction of the immersion $F\co V \looparrowright \R^4$ of a compact oriented $4$-manifold $V$ with boundary $\pa V=S^3$ to the boundary. The formula \eqref{eq:smale} is a generalized version of Hughes' formula.

\item By Poincar\'{e}-Hopf index theorem, if a singular Seifert surface $F\co V \to \R^4$ for an immersion $f\co S^3\looparrowright \R^4$ can be taken as an immersion, then the normal degree $\ndeg(f)$ of $f$ satisfies
\[
\ndeg(f)=\chi(V),
\]
where $\chi(V)$ denotes the Euler characteristic of $V$ (see \cite[Theorem~2.2~(b)]{k-m}).
\end{enumerate}
\end{rem}

\subsection{Stingley's formula and Thom polynomial for umbilic singularities}\label{sbsec:thom}
Let $g \co M\to N$ be a smooth map between oriented $4$-manifold whose rank $2$ points are all isolated. In \cite{stingley}, Stingley has studied singularities of smooth maps between oriented $4$-manifolds. In particular, he gave an integer $ind_p(g)$, called the index, associated to an isolated rank $2$ point $p\in M$, and defined the $4$-dimensional Riemann-Hurwitz type formula for maps with isolated rank $2$ singularities (see \cite[\S2]{stingley}, \cite[Definition~2.3]{e-t}). We denote by $\Sigma^{2}(g)$ the submanifold of rank $2$ points of $g$ and by $\sharp\Sigma^{2}(g)$ the sum of the indices of rank $2$ points of the map $g$. 
\begin{thm}
[Stingley {\cite[Theorem~I]{stingley}}]\label{thm:stin}
Let $g\colon M\to N$ be a smooth map between oriented $4$-manifolds
with isolated rank $2$ points and assume that $M$ is closed. Then
\begin{equation}\label{eq:sigma2long}
\bigl\langle g^*p_1(N)-p_1(M),[M]\bigr\rangle=\sharp\Sigma^{2}(g),
\end{equation}
where $p_1$ denotes the first Pontrjagin class and $[M]$ is the fundamental class of $M$.
\end{thm}

\begin{rem}\label{rem:umbilic}
It is known that a smooth generic map between oriented $4$-manifolds has only stable rank $2$ singularities, called the hyperbolic and elliptic umbilics. The index of a hyperbolic umbilic is $-1$, and that of an elliptic umbilic is $+1$. When $g \co M \to N$ is not generic, some umbilic points may degenerate, and non-umbilic rank $2$ point might appear. But, by perturbing $g$ at non-umbilic rank $2$ points, we can assume that all the isolated rank $2$ points of $g$ form the umbilic singularities. $\sharp\Sigma^{2}(g)$ is unchanged by perturbation of $g$.

Let $G\co W \to \R^4$ denote a generic map of a closed oriented $4$-manifold $W$, and let $\sharp\Sigma^{2,0}(G)$ denote the sum of the indices of the umbilic points of the map $G$. The Thom polynomial for umbilic singularities of $G$ is the special case of \eqref{eq:sigma2long}:
\begin{equation}\label{eq:thom}
\sharp\Sigma^{2,0}(G)=-\bigl\langle p_1(W),[W]\bigr\rangle=-3\sigma(W),
\end{equation}
the last equation is obtained by Hirzebruch signature theorem.
\end{rem}


For composition of maps, we have the following; 
\begin{lem}\label{lem:2com}
Let $f\co L\to M$ be a branched covering between compact oriented $4$-manifolds and let $g\co M\to \R^4$ be a smooth map. Assume that these maps are nonsingular near the boundaries and the rank $2$ points are isolated. If the intersection of the branched locus of $f$ and the singular set of $g$ is empty, then
\begin{equation}\label{eq:sigma2com}
\sharp\Sigma^{2}(g\circ f)=\deg f\cdot \sharp\Sigma^{2}(g)+\sharp\Sigma^{2}(f).
\end{equation}
\end{lem}
\begin{proof}
Let $S(f)$ and $S(g)$ be singular sets of $f$ and $g$. By the assumption $f \left( S(f)\right)  \cap S(g) = \emptyset$, we obtain
\begin{equation}
\Sigma^2(g\circ f)=f^{-1}\left( \Sigma^2(g)\right) \sqcup \Sigma^2(f).
\end{equation}
Then $f$ can be assumed to be a covering near $f^{-1}\left( \Sigma^2(g)\right)$. In particular, $f$ is a local diffeomorphism near each point $p \in f^{-1}\left(\Sigma^2(g)\right)$. For any $p \in \Sigma^2(g)$, the preimage of $p$ under $f$ is a finite set: 
\[
f^{-1}(p):=\{q_1, \dots q_k \}.
\]
If the orientation of the source manifold is changed then the index of each isolated rank $2$ point changes its sign (see definition of the index \cite[\S2]{stingley}, \cite[Definition~2.3]{e-t}). Thus we have
\[
\ind_{q_i}(g\circ f)=\pm \ind_{p}(g)
\]
where the sign is $+$ (resp.\ $-$) if $f$ preserves (resp.\ reverses) the orientations near $q_i$. Henceforth, $f^{-1}\left( \Sigma^2(g)\right)$ contributes to $\sharp\Sigma^{2}(g\circ f)$ by $\deg f\cdot \sharp\Sigma^{2}(g)$. 
\end{proof}

\section{Plumbing and Kirby diagrams}\label{sect:kirby}
We will use the following notation throughout this paper. The total space of the $2$-disk bundle over $S^2$ whose Euler class is equal to $-k$ will be denoted by $E(\xi_{-k})$, and the $2$-disk bundle over $\R P^2$ of Euler number $-k$ will be denoted by $E(\eta_{-k})$. Then an orientation-preserving double cover $\rho^* \colon E(\xi_{-2k})\to E(\eta_{-k})$ can be obtained by lifting the double cover $\rho \co S^2\to \R P^2$. It is known that the boundary of $E(\xi_{-k})$ is the lens space $L(k,1)$ and the boundary of $E(\eta_{-k})$ is $S^3/{\rm Dic}_k$ (see \cite[\S6.2]{orlik}, \cite[\S8]{h-n-k}). Let $M_{\circ}$ denote a complement of an open disk in a closed oriented $4$-manifold $M$.

\subsection{Kirby diagrams}
A {\em Kirby diagram} is a description of a $4$-dimensional $2$-handlebody consisting of one $0$-handle and a number of $1$- and $2$-handles. If the attaching maps of $1$- and $2$-handles are specified, then there exists essentially one and only one way to attach $3$-handles and $4$-handle to the $2$-handlebody to form a closed $4$-manifold. The attaching spheres of $1$- and $2$-handles are depicted by pair of $3$-balls and framed links in $S^3$, respectively. Often it will be convenient to depict $1$-handles by dotted circles as in Figure~\ref{fig:1-handle}. For the Kirby calculus refer to \cite{g-s}. 
\begin{figure}[]
\begin{center}
\includegraphics{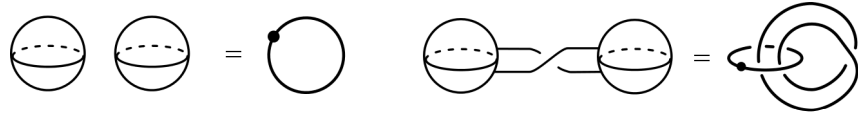}
\end{center}
\caption{Attaching spheres of $1$-handles}\label{fig:1-handle}
\end{figure}

\begin{defn}\label{def:k_move}
We denote by $W_{L}$ the $4$-dimensional $2$-handlebody represented by a framed link $L$ in $S^3$. Its boundary is denoted by $M_{L}:= \pa W_{L}$.
\end{defn}

The following two elementary operations on a framed link $L$, called {\em Kirby moves}, do not change the $3$-manifold $M_{L}$ (see \cite{kirby0}). 
\begin{enumerate}
\item[\bf{Move~{1}}] Adding to or deleting from $L$ an unknotted component with framing $\pm 1$ which is separated from the other components by a $2$-sphere. (see Figure~\ref{fig:k_move1})
\item[\bf{Move~{2}}] Replacing a component of $L$ with its connected sum with  a push-off of another component along the boundary of some band in $S^3$. (see Figure~\ref{fig:k_move2}, where $n_i$ denote the framing of $L_i$.)
\end{enumerate} 

\begin{rem}\label{rem:k_move}
\begin{enumerate}[(a)]
\item Adding a disjoint unknotted $\pm 1$ component to the link $L$  corresponds to taking connected sum of the $4$-manifold $W_{L}$ and $\C P^2$ or $\overline{\C P^2}$. This operation is called {\em blow-up} and the inverse operation is called {\em blow-down}.
\item Move~2 corresponds to changing the handlebody presentation by sliding one handle over another.
\end{enumerate}
\end{rem}

\begin{figure}[]
\begin{center}
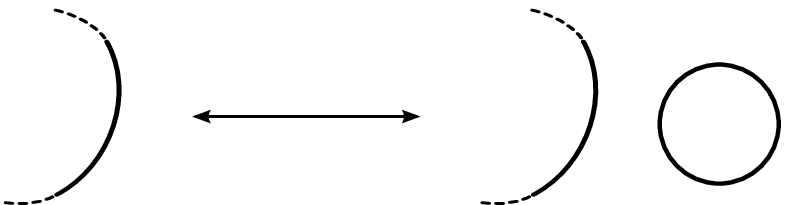
\end{center}
\caption{Move~1}\label{fig:k_move1}
\end{figure}

\begin{figure}[]
\begin{center}
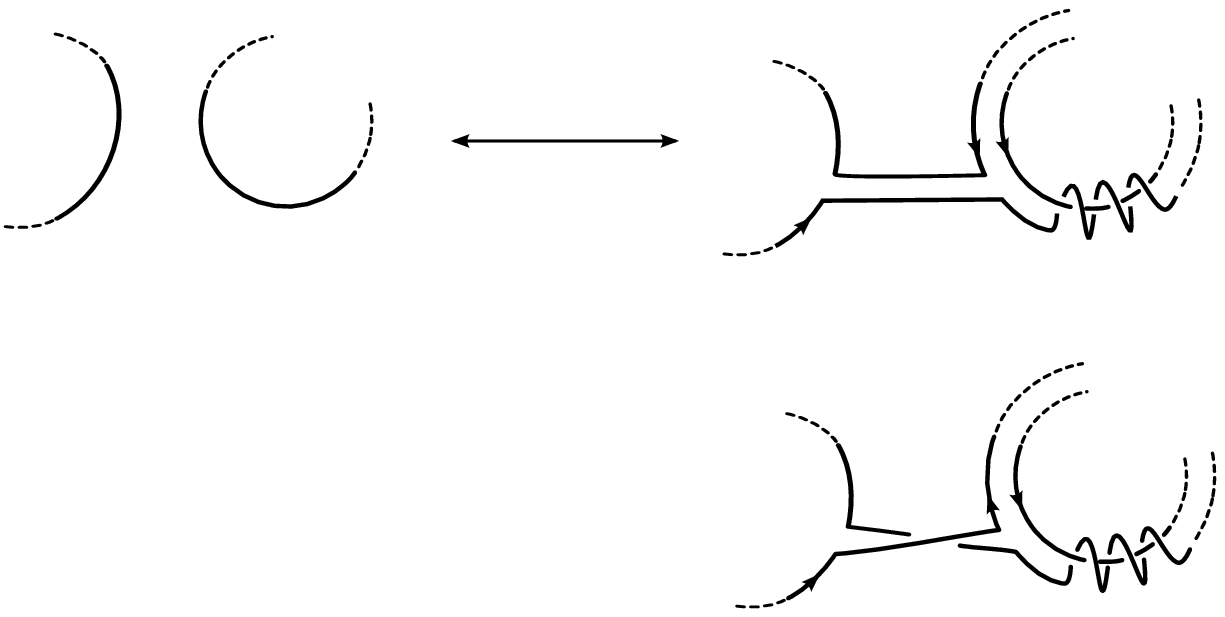
\end{center}
\caption{Move~2}\label{fig:k_move2}
\end{figure}

\begin{ex}\label{ex:handle_slide}
An unknotted circle with framing $\pm 1$ can be moved away from the rest of the link $L$ with the effect as drawn in Figure~\ref{fig:handle_slide}, by sliding arcs. In Figure~\ref{fig:handle_slide}, if the arcs belong to different components of $L$, their framings increase by $\mp 1$ each; if the arcs belong to a single component, the framing change by either $0$ or $\mp 4$.
\begin{figure}[]
\begin{center}
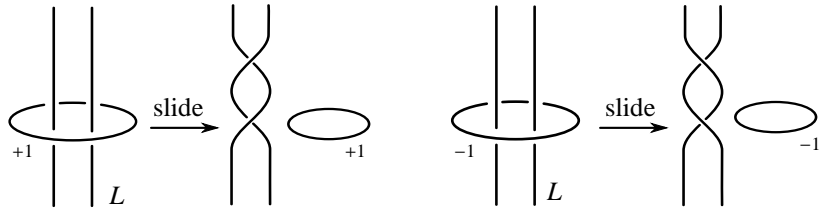
\end{center}
\caption{Handle sliding}\label{fig:handle_slide}
\end{figure}
\end{ex}

\begin{ex}\label{ex:k_move}
In Figure~\ref{fig:ex_kirby} (a), an unknotted circle with framing $k$ represents $E(\xi_{k})$. In particular, $\C P^2_{\circ}$ and $\overline{\C P^2}_{\circ}$ correspond to $k= \pm 1$. The $4$-manifolds $E(\eta_{k})$, $(S^2 \times S^2)_{\circ}$ and $(S^2 \tilde{\times} S^2)_\circ$ are drawn as Figure~\ref{fig:ex_kirby} (b), (c) and (d), respectively.
\begin{figure}[]
\begin{center}
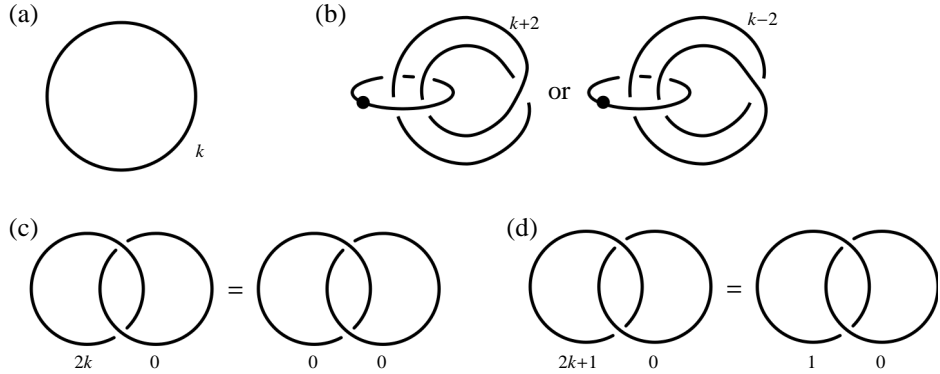
\end{center}
\caption{Examples; (a) $E(\xi_{k})$, (b) $E(\eta_{k})$, (c) $(S^2 \times S^2)_{\circ}$ and (d) $(S^2 \tilde{\times} S^2)_\circ$}\label{fig:ex_kirby}
\end{figure}
\end{ex}

\subsection{Coverings}\label{subsec:cover}
In some cases, a branched covering $X^4 \to Y^4$ can be picturized by using Kirby diagrams (see \cite[\S6.3]{g-s}). One of the simplest cases is when $f$ is  an unramified $\Z_d$-covering and $Y$ is drawn as in Figure~\ref{fig:covering_2}. 
In this case, we identify $D^4\cup 1$-handle with $S^1 \times D^3$, and there is an obvious $d$-fold covering $S^1 \times D^3 \to S^1 \times D^3$ given by $(z,w)\mapsto (z^d,w)$. Each of the remaining handles of $Y$ lifts to $d$ handles of $X$ (Figure~\ref{fig:covering_2} is in the case of $d=3$). The blackboard framing on each knot lift to the blackboard framing in each of the $d$ corresponding knots in the diagram for $X$ (see \cite[\S 6.3]{g-s}).
\begin{figure}[]
\begin{center}
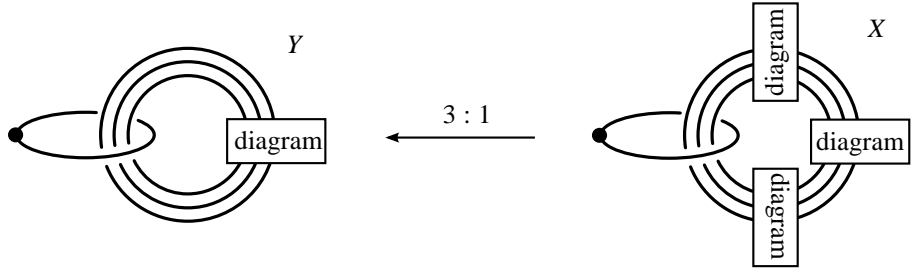
\end{center}
\caption{Kirby diagrams representing a $\Z_3$-covering $X \to Y$}\label{fig:covering_2}
\end{figure}

\begin{ex}\label{ex:covering}
Let $Y$ be a $D^2$-bundle over $\R P^2$ with Euler number $-k$, drawn as in Figure~\ref{fig:ex_kirby} (b). The double cover $X$ over $Y$ is $E(\xi_{-2k})$, the $D^2$-bundle over $S^2$ with Euler class $-2k$ (see Figures~\ref{fig:covering_dic} and \ref{fig:covering_dic_2}). We call the $1$-handle with a $2$-handle attached, drawn in Figure~\ref{fig:covering_dic_2} (d), a {\em cancelling pair}. Since the core of the $1$-handle and the belt sphere of the $2$-handle intersect transversally at only one point, these handles cancel out with each other. 
\begin{figure}[]
\begin{center}
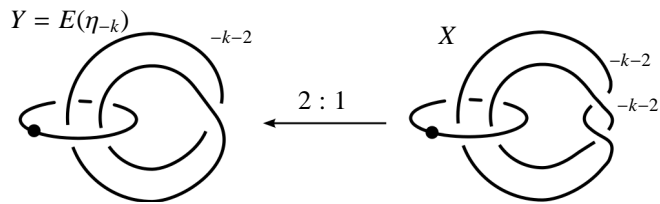
\end{center}
\caption{A $\Z_2$-cover $X$ over a $D^2$-bundle over $\R P^2$}\label{fig:covering_dic}
\end{figure}
\begin{figure}[]
\begin{center}
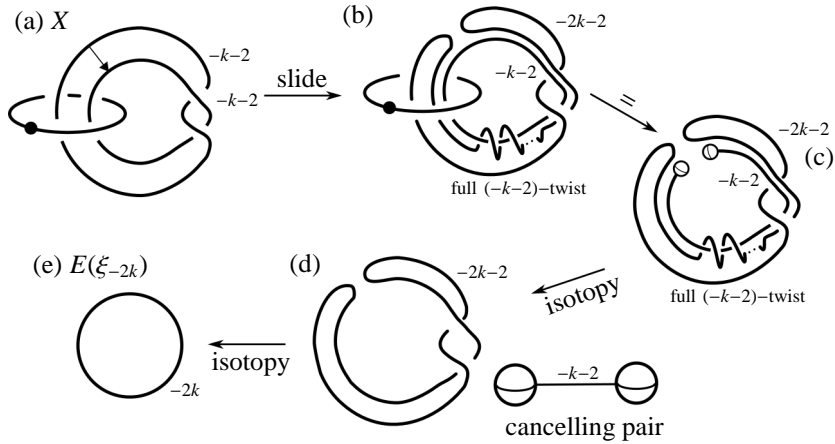
\end{center}
\caption{$X=E(\xi_{-2k})$}\label{fig:covering_dic_2}
\end{figure}
\end{ex}

\begin{ex}\label{ex:covering_2}
Let $\Pi_k\co E(\xi_{-1})\to E(\xi_{-k})$ be a map of $D^2$-bundles over $S^2$, given by $z \mapsto z^k$ on each fiber (see \cite[\S3]{e-t}). Let $\Phi \co E(\xi_{-2k})\to E(\eta_{-k})$ be the $\Z_2$-covering defined in Example~\ref{ex:covering}. Composite of $\Phi \co E(\xi_{-2k})\to E(\eta_{-k})$ with $\Pi_{2k}\co E(\xi_{-1})\to E(\xi_{-2k})$ is a branched $\dic_k$-covering over $E(\eta_{-k})$.

 In fact, a dicyclic group $\dic_k$ can be defined as the subgroup of the group of unit quaternions generated by
\[
a=e^{\frac{i \pi}{k}},~x=j.
\]
In addition, the dicyclic group $\dic_k$ is an extension of the cyclic group of order $2$ by a cyclic group of order $2k$. This extension can be expressed as the exact sequence
\begin{equation}\label{eq:exact}
\xymatrix{
\{1\}\ar[r]&\Z_{2k}\ar[r]&\dic_k\ar[r]&\Z_2\ar[r]&\{1\},
}
\end{equation}
where the map 
\[
\Z_{2k} \longrightarrow \dic_k=\bigl\langle a,x \mid a^k=x^2,xax^{-1}=a^{-1}\bigr\rangle
\]
is defined by sending the generator of $\Z_{2k}$ to the first generator $a$ of $\dic_n$. Since $\Z_{2k}$ is a normal subgroup of $\dic_k$ and since the sequences \eqref{eq:exact} is exact, a $\Z_2$-covering 
\begin{equation}\label{eq:Z_2}
S^3/\Z_{2k}=L(2k,1) \xrightarrow{\Z_2 \text{-covering}} S^3/\dic_k
\end{equation}
is uniquely defined. A $\dic_k$-covering $S^3 \to S^3/\dic_k$ is represented by the composite of the $\Z_{2k}$-covering and $\Z_2$-covering. Restricting $\Phi$ to the boundary $\pa E(\xi_{-2k})=L(2k,1)$ we obtain the same $\Z_2$-covering as \eqref{eq:Z_2}. Thus, the following diagram is commutative; 
\[\xymatrix{
S^3=L(1,1)=\pa E(\xi_{-1})\ar[rr]^-{\text{$\Z_{2k}$-covering}}\ar[rrd]_-{\text{$\dic_k$-covering}}&&L(2k,1)=\pa E(\xi_{-2k})\ar[d]^-{\rho^*}\ar[rr]^-{\text{$S^1$-bundle}}&&S^2\ar[d]^-{\rho}\\
&&S^3/\dic_k = \pa E(\eta_{-k}) \ar[rr]^-{\text{$S^1$-bundle}}&&{\R}P^2,
}\]
where $\rho^*$ is an orientation-preserving double cover obtained by lifting the double cover $\rho \co S^2 \to \R P^2$. 
\end{ex}

\subsection{Plumbing}
\begin{defn}\label{def:plum}
Let $E_1\to S^n_1$ and $E_2\to S^n_2$ be two oriented $n$-disk bundles over $S^n$. Let $D^n_i \subset S^n_i$ be embedded $n$-disks in the base spaces and let
\[
\varphi_i \co D^n_i \times D^n \to E_i|_{D_i^n}
\]
be trivializations of the restricted bundles $E_i|_{D_i^n}$ for $i=1,2$. A \textit{plumbing} of $E_1$ and $E_2$ is a $2n$-dimensional manifold obtained by identifying the points $\varphi_1(x,y)$ and $\varphi_2(y,x)$ for each $(x,y)\in D^n \times D^n$ in the disjoint union of $E_1$ and $E_2$. A plumbing is a smooth manifold with corners. However, there is a canonical way to smooth the corners, so we may assume a plumbing is a smooth $4n$-manifold with boundary.
\end{defn}

Now, we consider $4$-dimensional plumbing. Let $(A_n,m_1,\dots,m_n)$ and $(D_n,m_1,\dots,m_n)$ stand for the following weighted Dynkin diagrams:
\[
\xygraph{
\bullet ([]!{+(0,-.3)} {m_1}) - [r]
\bullet ([]!{+(0,-.3)} {m_2}) - [r] \cdots - [r]
\bullet ([]!{+(0,-.3)} {m_{n - 1}}) - [r]
\bullet ([]!{+(0,-.3)} {m_n})
}
\]
\[
\xygraph{
\bullet ([]!{+(0,-.3)} {m_1}) - []!{+(1,-.5)}
\bullet ([]!{+(0,-.3)} {m_3}) (
- []!{+(-1,-.5)} \bullet ([]!{+(0,-.3)} {m_2}),
- [r] \bullet ([]!{+(0,-.3)} {m_4})- [r] \dots - [r]
\bullet ([]!{+(0,-.3)} {m_{n-1}}) - [r]
\bullet ([]!{+(0,-.3)} {m_{n}}))
}
\]
\begin{defn}\label{def:plum2}
{\em The plumbing of $E(\xi_{m_1}), \dots ,E(\xi_{m_n})$ along the weighted Dynkin diagram $(A_n,m_1,\dots,m_n)$ (or $D_n$)} is the $4$-manifold obtained by plumbing $E(\xi_{m_i})$ in such a way as the edges of $A_n$ (or $D_n$) indicate. The resulting manifold is denoted as $P(A_n,m_1,\dots,m_n)$ (or $P(D_n,m_1,\dots,m_n)$).
\end{defn}
We can draw the Kirby diagrams representing any plumbing of $E(\xi_{m_1}), \dots ,E(\xi_{m_n})$ along a weighted Dynkin diagram (see \cite[\S4.6]{g-s}). For instance, the plumbing along the weighted Dynkin diagram $(D_5,m_1,\dots,m_n)$ is drawn in Figure~\ref{fig:D_5}.
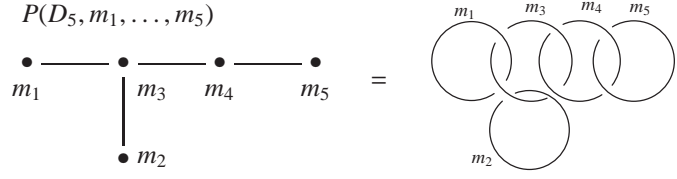
\begin{figure}[]
\begin{center}
\input{D_5}
\end{center}
\caption{$D_5$-plumbing}\label{fig:D_5}
\end{figure}

\section{Proof of main theorem}\label{sect:proof}
We recall the construction of immersions $f_n$ and $g_n$. Consider the immersion of $E(\xi_{2})$ into $\R^4$ \cite{e-t}. Plumbing copies of this immersion along $A_{n-1}$ or $D_{n+2}$, we obtain immersions $P(A_{n-1},2):=P(A_{n-1},2,\cdots,2)$ and $P(D_{n+2},2):=P(D_{n+2},2,\cdots,2)$ into $\R^4$. Restricting them to the boundaries, we obtain immersions:
\begin{align*}
\kappa_{{A},n}\co & \pa P(A_{n-1},2)=L(n,1) \lo \R^4,\\
\kappa_{{D},n}\co & \pa P(D_{n+2},2)=S^3/\dic_n \lo \R^4.
\end{align*}
The immersions $f_n$ and $g_n$ are the composites of the universal coverings with $\kappa_{{A},n}$ and $\kappa_{{D},n}$:
\begin{align*}
f_n \co & S^3 \twoheadrightarrow L(n,1) \looparrowright \R^4,\\
g_n \co & S^3 \twoheadrightarrow S^3/\dic_n \looparrowright \R^4.
\end{align*}
In the proof, we use Kirby diagrams to construct manageable singular Seifert surfaces (that is, maps whose umbilic points can be easily counted) for immersions from $S^3$ into $\R^4$.

\subsection{Type $A$}\label{subsec:A}
The proof is divided into three steps:
\begin{enumerate}[(1)]
\item We construct an immersion $f_n' \co S^3 \looparrowright \R^4$ which is obtained by modifying the immersion $f_n$ and for which a manageable singular Seifert surface can be constructed.
\item We compute $\Omega(f_n')$.
\item To compute $\Omega(f_n)$, we find the difference of $\Omega(f_n)$ and $\Omega(f_n')$.
\end{enumerate}

\textit{Step~1.}
Since constructing a manageable singular Seifert surface for $f_n$ is difficult (see Remark~\ref{rem:a}), instead we construct a singular Seifert surface for an immersion $f_n'$ that is obtained by modifying $f_n$. Let $\na$ denote the boundary connected sum. The $4$-manifold $P(A_{n-1},2) \na (S^2 \times S^2)_{\circ}$ is diffeomorphic to the boundary connected sum of $E(\xi_{-n})$ and $n$ copies of $\C P^2_{\circ}$ (see Figure~\ref{fig:a_01});
\begin{equation}\label{eq:a_01}
P(A_{n-1},2) \na (S^2 \times S^2)_{\circ} \approx E(\xi_{-n})\na (\C P^2_{\circ})^{\na n}.
\end{equation}

\begin{figure}[]
\begin{center}
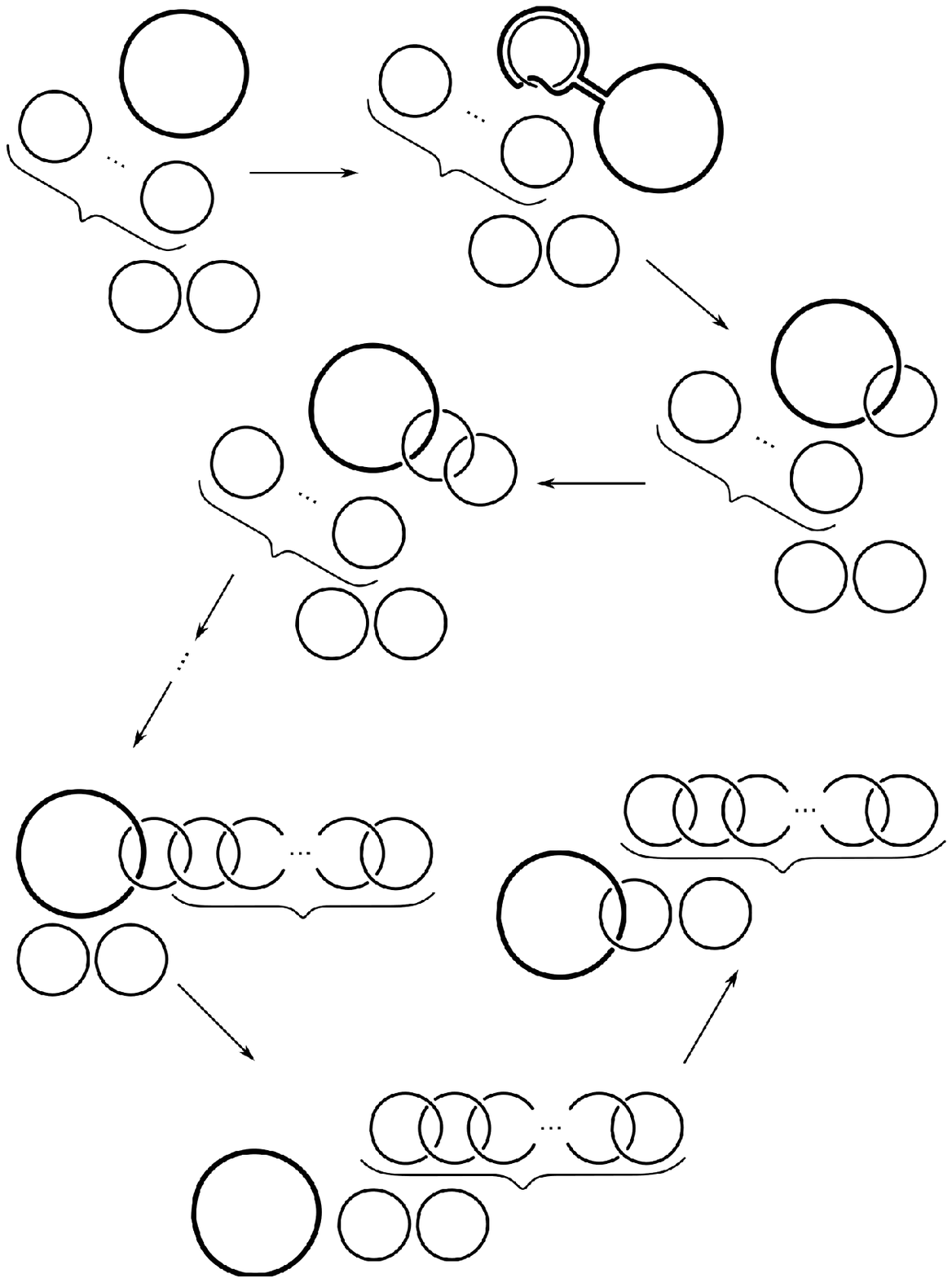
\end{center}
\caption{Proof of \eqref{eq:a_01}}\label{fig:a_01}
\end{figure}

In Example~\ref{ex:covering_2}, we defined the map $\Pi_n\co E(\xi_{-1})\to E(\xi_{-n})$. The map $\Pi_n$ is a branched $\Z_n$-covering with the zero-section of $E(\xi_{-n})$ being its branched locus: we can define the branched $\Z_n$-covering
\[
\Pi'_{A,n} \co E(\xi_{-1})\na (\C P^2_{\circ})^{\na n^2} \to E(\xi_{-n})\na (\C P^2_{\circ})^{\na n}
\]
as the connected sum of $\Pi_n$ and the trivial $\Z_n$-covering $(\C P^2_{\circ})^{\na n^2}\to (\C P^2_{\circ})^{\na n}$. It is easy to see that $(S^2 \times S^2)_{\circ}$ can be immersed in $\R^4$ and we choose one of them, denoted as $T \co (S^2 \times S^2)_{\circ}\looparrowright \R^4$. Let 
\[
K'_{A,n}\co P(A_{n-1},2) \na (S^2 \times S^2)_{\circ} \lo \R^4
\]
be an immersion defined as the connected sum of the immersion $K_{A,n}\co P(A_{n-1},2) \looparrowright \R^4$ constructed in the beginning of this section and $T \co (S^2 \times S^2)_{\circ}\looparrowright \R^4$. We obtain the map $F_n' \co E(\xi_{-1})\na (\C P^2_{\circ})^{\na n^2} \to \R^4$ by composing $\Pi'_{A,n}$ and $K'_{A,n}$;
\begin{multline*}
F_n' := K'_{A,n} \circ \Pi'_{A,n} \co E(\xi_{-1})\na (\C P^2_{\circ})^{\na n^2} \xrightarrow{\Pi'_{A,n}}  E(\xi_{-n})\na (\C P^2_{\circ})^{\na n}\\
 \approx P(A_{n-1},2) \na (S^2 \times S^2)_{\circ} \stackrel{K'_{A,n}}{\looparrowright}  \R^4.
\end{multline*}
Since the boundary of $(S^2 \times S^2)_{\circ}$ is $S^3$, $\pa \left(P(A_{n-1},2)\na (S^2 \times S^2)_{\circ}\right)$ is $L(n,1)$. Therefore $F_n'$ is a singular Seifert surface for the immersion 
\[
f_n' \co S^3 \to \pa \left(P(A_{n-1},2)\na (S^2 \times S^2)_{\circ}\right)=L(n,1) \looparrowright \R^4.
\]
The situation is summarized in the following commutative diagram; 
\[\xymatrix{
S^3 \ar@/^9mm/[rrrr]^(.65){f'_n} \ar@{^{(}->}[d] \ar[rr]^(.35){\text{universal covering}}&&\pa \left(P(A_{n-1},2)\na (S^2 \times S^2)_{\circ}\right)=L(n,1)\ar@{^{(}->}[d] \ar[rr]^(.65){\kappa'_{A,n}}&&\R^4\\
E(\xi_{-1})\na (\C P^2_{\circ})^{\na n^2} \ar@/^-20mm/[rrrru]_(.65){F'_n} \ar[rr]_ (.35){\Pi'_{A,n}}&&E(\xi_{-n})\na (\C P^2_{\circ})^{\na n} \approx P(A_{n-1},2) \na (S^2 \times S^2)_{\circ}\ar[rru]_(.65){K'_{A,n}}
}\]
The immersion $f_n'$ is defined as a modification of the immersion $f_n$. More precisely we replace $\kappa_{{A},n}$ with $\kappa'_{{A},n} := \kappa_{A,n} \# t$.

\begin{rem}\label{rem:a}
In Figure~\ref{fig:a_01}, the transformation from (a) to (f) proves 
\begin{equation}\label{eq:a_02}
E(\xi_{-n})\na (\C P^2_{\circ})^{\na (n-1)} \approx P(A_{n-1},2)\na \overline{\C P^2}_{\circ}.
\end{equation}
Since there is a covering $E(\xi_{-1})\to E(\xi_{-k})$, it seems to be possible that we can obtain a covering over $P(A_{n-1},2)$ by deleting one $\overline{\C P^2}$ from \eqref{eq:a_02}, but it is not a ``good'' covering. The reason for that is that the branch locus of the branched covering $E(\xi_{-1})\to E(\xi_{-n})$ is included in $\overline{\C P^2}$ corresponding to the unknotted component with framing $-1$ (the thickened circle in Figure~\ref{fig:a_01} (f)), and because $\overline{\C P^2}_{\circ}$ cannot be immersed in $\R^4$ the branched locus of $E(\xi_{-1})\to E(\xi_{-n})$ should intersect the singular set of the map $\overline{\C P^2}_{\circ} \to \R^4$, and it would become difficult to examine indices of rank $2$ points of singular Seifert surface $E(\xi_{-1})\na (\C P^2_{\circ})^{\na n(n-1)} \to E(\xi_{-n})\na (\C P^2_{\circ})^{\na (n-1)} \approx P(A_{n-1},2)\na \overline{\C P^2}_{\circ}\to \R^4$. In contrast, as we will see in step 2, it is possible to count the umbilic points of the singular Seifert surface $F_n'$ for $f_n'$.

Consider the connected sum of $P(A_{n-1},2)\na \overline{\C P^2}_{\circ}$ and two $\C P^2$'s. Since $\C P^2 \# \C P^2 \# \overline{\C P^2}$ is diffeomorphic to $(S^2 \times S^2) \# \C P^2$ (see \cite[Corollary~4.3]{kirby}), we obtain Figure~\ref{fig:a_01} (g). Then deleting one $\C P^2$ in (a) and (g), we have \eqref{eq:a_01}.
\end{rem}

\textit{Step~2.}
We compute the Smale invariant of $f_n'$. By Remark~\ref{rem:hughes} (b), the normal degree of the immersion $\kappa'_{A,n} \co \pa \left( P(A_{n-1},2)\na (S^2 \times S^2)_{\circ} \right)\looparrowright \R^4$ equals
\[
\chi\left(P(A_{n-1},2)\na (S^2 \times S^2)_{\circ}\right)=\chi\left(P(A_{n-1},2)\# (S^2\times S^2)\right)=n+2.
\]
Since $S^3 \to L(n,1)$ has degree $n$, the composite $f_n' \co S^3 \to \pa P(A_{n-1},2)\# S^3=L(n,1) \looparrowright \R^4$ has normal degree
\[
\ndeg(f_n')=n(n+2)
\]
(see also \cite[\S{IV}]{milnor}, \cite[Proof of Theorem~1.1]{e-t}).

The signature of $E(\xi_{-1})\na (\C P^2_{\circ})^{\na n^2}$ is
\[
\sigma\left(E(\xi_{-1})\na (\C P^2_{\circ})^{\na n^2}\right)=\sigma\left(\overline{\C P^2}_{\circ} \na (\C P^2_{\circ})^{\na n^2}\right)=n^2-1
\]
by Novikov additivity. Furthermore, since $K'_{A,n} \co P(A_{n-1},2)\na (S^2 \times S^2)_{\circ} \lo \R^4$ is an immersion and since $\Pi'_{A,n}$ is an unramified covering if restricted over $\C P^2_{\circ}$ (namely, there are no rank $2$ points on the copies of $\C P^2_{\circ}$), the algebraic number of rank $2$ points of the map $F_n'=K'_{A,n} \circ \Pi'_{A,n}$ is
\[
\sharp\Sigma^2(F_n')=\sharp\Sigma^2(K'_{A,n} \circ \Pi'_{A,n})=\sharp\Sigma^2(\Pi'_{A,n})=-(n^2-1)
\]
here the last equality follows from \cite[\S3,~Remark~2.4]{e-t} (in \cite{e-t}, the map ${\Pi'_{A,n}}|_{E(\xi_{-1})}$ is denoted as $\Pi_n$). Thus the Hirzebruch defect of $f_n'$ is
\[
\hdef(f_n')=-3\sigma \left(E(\xi_{-1})\na (\C P^2_{\circ})^{\na n^2}\right)-\sharp\Sigma^2(F_n')=-2n^2+2
\]
by \eqref{eq:hirzebruch}. Substituting these values into the formula \eqref{eq:smale}, we have
\[
\Omega(f_n')=\left(n^2+2n-1, \frac{-(-2n^2+2)-2(n^2+2n-1)}{4} \right)=(n^2+2n-1,-n).
\]

\textit{Step~3.}
To determine the Smale invariant of $f_n$, we compute the difference of Smale invariants of $f_n$ and $f_n'$. 

\begin{lem}\label{lem:step3}
Let $M$ be a $3$-manifold of the form $S^3/\Gamma$ where $\Gamma$ is a finite subgroup of $S^3$ and let $\kappa \co S^3/\Gamma=M \lo \R^4$ be an immersion. Let $\pi \co S^3 \to M$ denote the universal $\Gamma$-cover which has degree $n=\abs{\Gamma}$. Consider the following two immersions $S^3 \lo \R^4$:
\begin{enumerate}[(a)]
\item 
Composite of $\pi$ with $\kappa$;
\[
\alpha := \kappa \circ \pi \co S^3\to S^3/\Gamma=M\looparrowright \R^4.
\]

\item 
Composite of the connected sum of $\pi$ and $n$ copies of identity map of $S^3$ with the connected sum of $\kappa$ and another immersion $\beta \co S^3 \lo \R^4$;
\[
\alpha ':= (\kappa \# \beta)\circ (\pi \# (id)^{\# n} ) \co  S^3\# (S^3)^{\# n} \to (S^3/\Gamma) \# S^3 \lo \R^4.
\]
\end{enumerate}
Then, the difference of the Smale invariant of $\alpha$ and $\alpha'$ is
\[
\Omega(\alpha)-\Omega(\alpha')=n\cdot \Omega(\beta).
\]
\end{lem}
This holds because $\Omega$ is additive under connected sums. 

Recall that $t \co S^3 \looparrowright \R^4$ denotes the restriction of the immersion $T \co (S^2\times S^2)_{\circ} \looparrowright \R^4$ to the boundary. The Smale invariant $\Omega(t)$ of $t$ can be computed by the formula \eqref{eq:smale}. Since $(S^2\times S^2)_{\circ} \looparrowright \R^4$ is an immersion, the normal degree and the Hirzebruch defect of $t$ are
\[
\ndeg(t)=\chi\left((S^2\times S^2)_{\circ}\right)=3,
\]
\[
\hdef(t)=-3\sigma \left((S^2\times S^2)_{\circ}\right)-\sharp\Sigma^2\left((S^2\times S^2)_{\circ} \looparrowright \R^4\right)=0
\]
see Remark~\ref{rem:hughes}~(b). Therefore, the Smale invariant of $t$ is
\[
\Omega(t)=(2,-1).
\]
Choose an immersion $s \co S^3 \looparrowright \R^4$ of Smale invariant $\Omega(s)=(-2,1)$ (namely $s=-t \in \Imm[S^3,\R^4]$). We will apply Lemma~\ref{lem:step3} to the case where $\alpha =f'_n$ and $\beta =s$.

Let $\pi'_n$ denote the connected sum of $\Z_n$-covering $S^3 \to L(n,1)$ and $n$ copies of identity map of $S^3$;
\[
\pi_n' \co S^3\#(S^3)^{\#n} \to L(n,1)\# S^3.
\]
Composing $\pi_n'$ with $(\kappa_{{A},n}\# t)\# s$ we obtain an immersion 
\[
f_n '' :=\left((\kappa_{{A},n}\# t)\# s\right)\circ \pi_n' \co  S^3\#(S^3)^{\#n} \to L(n,1)\# S^3 \lo \R^4,
\]
which corresponds to the immersion $\alpha'$ in Lemma~\ref{lem:step3}.
Since $s=-t$, the immersion $(\kappa_{{A},n}\# t)\# s \co \left(L(n,1)\# S^3\right)\# S^3 \looparrowright \R^4$ is regularly homotopic to the immersion $\kappa_{{A},n} \co L(n,1) \looparrowright \R^4$. In addition, $\pi_n' \co S^3\#(S^3)^{\#n} \to L(n,1)\# S^3$ is regularly homotopic to a $\Z_n$-cover, so the immersion $f_n''$ is regularly homotopic to the immersion $f_n$ which we constructed in Section~\ref{sec:intr} via the plumbing. Therefore, by Lemma~\ref{lem:step3},
\[
\Omega(f_n)=\Omega(f_n'')=(n^2+2n-1,-n)+n(-2,1)=(n^2-1,0).
\]

\subsection{Type $D$}

We can compute the Smale invariant of $g_n$ in a similar way as for $f_n$.

\textit{Step~1.}
We construct a singular Seifert surface for the immersion that is obtained by modifying $g_n$ (see Remark~\ref{rem:d} (a)). Let $E_{-n}^*$ denote the $4$-manifold obtained by replacing a $1$-handle with a $2$-handle with framing $0$ in the Kirby diagram representing $E(\eta_{-n})$ (see Figure~\ref{fig:ex_kirby} (b)). Since this replacement corresponds to replacing $S^1 \times D^3$ with $D^2 \times S^2$, it does not change the boundary. The $4$-manifold $P(D_{n+2},2) \na (S^2 \times S^2)_{\circ}$ is represented as the boundary connected sum of $E_{-n}^*$ and $(n+2)$ copies of $\C P^2_{\circ}$ (see Figure~\ref{fig:dic_1});
\begin{equation}\label{eq:d_01}
P(D_{n+2},2) \na (S^2 \times S^2)_{\circ} \approx E_{-n}^*\na (\C P^2_{\circ})^{\na (n+2)}.
\end{equation}
\begin{figure}[]
\begin{center}
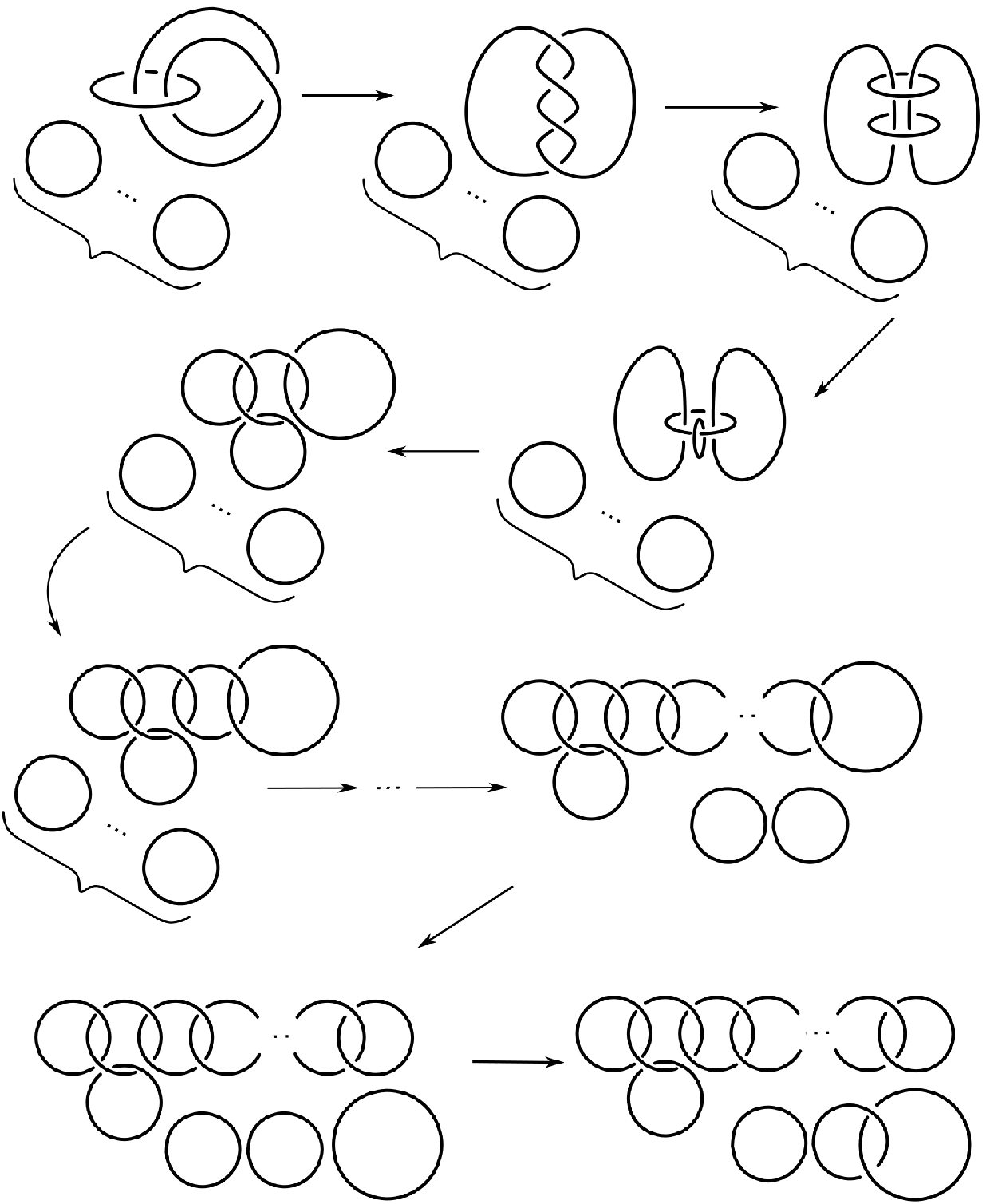
\end{center}
\caption{Proof of \eqref{eq:d_01}}\label{fig:dic_1}
\end{figure}
By Example~\ref{ex:covering_2}, we obtained the $\dic_n$-covering
\begin{equation}
\Pi_n \circ \Phi |_{\pa} \co S^3 \to L(2n,1) \to S^3/\dic_n.
\end{equation}
Thus, we construct a branched $\Z_2$-covering over $E_{-n}^* \na (\C P^2_{\circ})^{\na (n+2)}$ defined by an extension of the $\dic_n$-covering $\Pi_n \circ \Phi |_{\pa}$. In order to construct the $\Z_2$-covering, we will use the following Lemma:

\begin{lem}\label{lem:a}
We can extend a $\Z_2$-covering $\Phi|_{\pa} \co L(2n,1)\to S^3/\dic_n$ to a branched $\Z_2$-covering $X'\to E_{-n}^*$, where $X'$ is given by the following Kirby diagram; 

\[
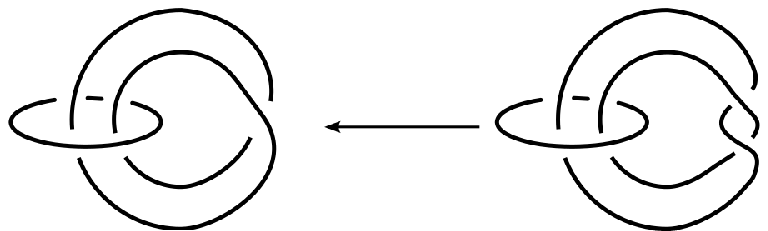
\]
\end{lem}
\begin{proof}
The above diagrams of $X'$ and $Y'$ are given by replacing $1$-handles of $X$ and $Y$ in Figure~\ref{fig:covering_dic} with $2$-handles with framing $0$. Since this operation dose not change boundaries, we obtain the same $\Z_2$-covering $\pa X' \to \pa Y'=S^3/\dic_n$ as $\Phi|_{\pa}$. It only remains to extend this map in the interior. For the inside of two $2$-handles with framing $-n-2$ of $X'$, the extension is defined by trivial double covering over the $2$-handle with framing $-n-2$ of $Y'$ in. In addition, for the part of $2$-handle with framing $0$ of $X'$, the extension is defined by the map $D^2 \times S^2 \to D^2 \times S^2$, $(z,w)\mapsto (z^2,w)$ (see the right side of Figure~\ref{fig:k}). This extension is well-defined on the attaching part.
\end{proof}

\begin{lem}\label{lem:b}
Let $X_n$ denote the $(S^2 \times S^2)_{\circ}$ or $(S^2 \tilde{\times} S^2)_\circ$ according to whether $n$ is even or odd (see Example~\ref{ex:k_move} (c) and (d)). The $\Z_2$ cover over $E_{-n}^*$ is diffeomorphic to $E(\xi_{-2n})\na X_n$.
\end{lem}
\begin{proof}
The proof is given by Figure~\ref{fig:lemma_b}. In the second step we use the fact that a component which is Hopf-linked by an unknotted circle with framing $0$ can be separated from other components by a sequence of handle slides (see \cite[Lemma~4.5]{kirby}).
\begin{figure}[]
\begin{center}
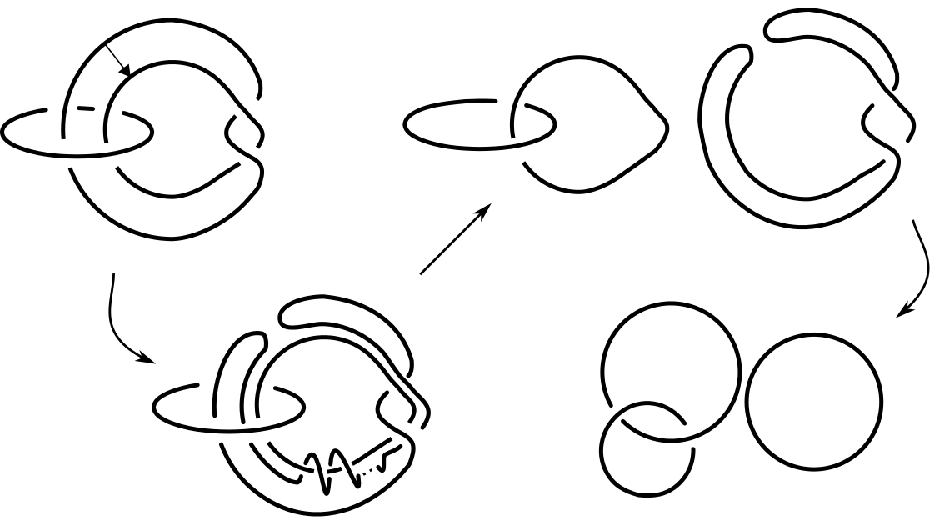
\end{center}
\caption{Proof of Lemma~\ref{lem:b}}\label{fig:lemma_b}
\end{figure}
\end{proof}

By Lemmas~\ref{lem:a} and \ref{lem:b}, we obtain the branched $\Z_2$-covering
\[
\tilde{\Phi}_1 \co E(\xi_{-2n}) \na X_n \to  E_{-n}^* 
\]
that defines an extension of the $\Z_2$-covering $\pa E(\xi_{-2n})\to \pa E(\eta_{-n})=\pa E_{-n}^*$.

Next, we consider the branched locus of $\tilde{\Phi}_1$. Let $h_1$ denote the $2$-handle with framing $0$ in the Kirby diagram representing $E_{-n}^*$.
\begin{lem}\label{lem:c}
The branched locus of $\tilde{\Phi}_1$ corresponds to the core of the $2$-handle $h_1$. 
\end{lem}
\begin{proof}
We identify $D^4 \cup h_1$ with $D^2 \times S^2$. In \S\ref{subsec:cover}, the covering map $\Phi$ is defined by $(z,w)\mapsto (z^2,w)$ over $S^1 \times D^3 = D^4 \cup 1$-handle in $X$. Thus the map $\tilde{\Phi}_1|_{\pa}= \Phi|_{\pa}$ is defined by $(z,w)\mapsto(z^2,w)$ over $\pa D^2 \times S^2=S^1 \times S^2$, and $\tilde{\Phi}_1 \co D^2 \times S^2 \to D^2 \times S^2$ is also given by $(z,w)\mapsto(z^2,w)$ (see Figure~\ref{fig:k}). The rank of $d \tilde{\Phi}_1$ equals $4$ outside the $0$-section $\{0\} \times S^2$ in $D^2 \times S^2$ and equals $2$ at any point on the $\{0\} \times S^2$. 
\begin{figure}[]
\begin{center}
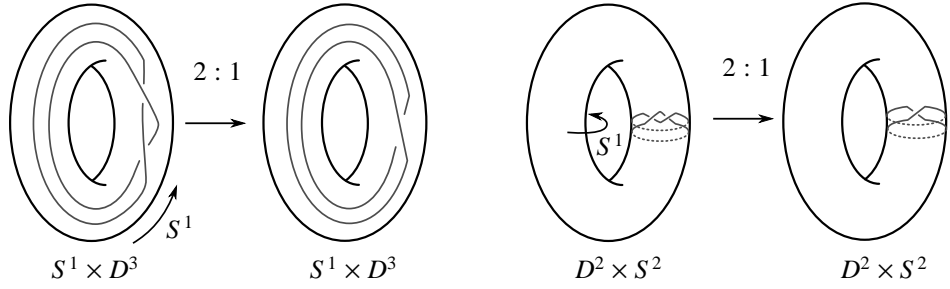
\end{center}
\caption{$\Z_2$-coverings and attaching spheres of $1$-handle and $2$-handle with framing $0$}\label{fig:k}
\end{figure}
\end{proof}

Let $\Phi_1$ denote a branched covering over $E_{-n}^* \na (\C P^2_{\circ})^{\na (n+2)}$ which is defined by taking connected sum of $\tilde{\Phi}_1$ with the trivial $\Z_2$-covering $(\C P^2_{\circ})^{\na 2(n+2)} \to (\C P^2_{\circ})^{\na (n+2)}$; 
\[
\Phi_1 \co E(\xi_{-2n}) \na X_n \na (\C P^2_{\circ})^{\na 2(n+2)} \to  E_{-n}^* \na (\C P^2_{\circ})^{\na (n+2)}.
\]
Let 
\[
\Phi_2 \co E(\xi_{-1})\na (X_n)^{\na 2n}\na (\C P^2_{\circ})^{\na 4n(n+2)} \to E(\xi_{-2n})\na X_n \na (\C P^2_{\circ})^{\na 2(n+2)}
\]
denote the $2n$-fold branched covering obtained by taking the connected sum of the covering $\Pi \co E(\xi_{-1})\to E(\xi_{-2n})$ with the trivial $2n$-fold covering $(X_n)^{\na 2n}\na (\C P^2_{\circ})^{\na 4n(n+2)} \to (X_n)\na (\C P^2_{\circ})^{\na 2(n+2)}$. We obtain the branched $\dic_n$-covering
\[
\Pi'_{D,n} \co E(\xi_{-1})\na (X_n)^{\na 2n}\na (\C P^2_{\circ})^{\na 4n(n+2)} \xrightarrow{\Phi_2} E(\xi_{-2n})\na (X_n) \na (\C P^2_{\circ})^{\na 2(n+2)} \xrightarrow{\Phi_1} E_{-n}^* \na (\C P^2_{\circ})^{\na (n+2)},
\]
which is defined by composing $\Phi_1$ and $\Phi_2$ (see Remark~\ref{rem:d} (b)). 


Recall that $T\co (S^2 \times S^2)_{\circ} \lo \R^4$ denotes the immersion that has been chosen in \S\ref{subsec:A}. Let $K'_{D,n}$ be an immersion defined as the connected sum of the immersion $K_{D,n}\co P(D_{n+2},2) \looparrowright \R^4$ (constructed in the beginning of this section) and $T \co (S^2 \times S^2)_{\circ}\looparrowright \R^4$. We obtain the map $G_n' \co E(\xi_{-1})\na (X_n)^{\na 2n}\na (\C P^2_{\circ})^{\na 4n(n+2)} \to \R^4$ by composing $\Pi'_{D,n}$ and $K'_{D,n}$;
\begin{multline*}
G_n' := K'_{D,n} \circ \Pi'_{D,n} \co E(\xi_{-1})\na (X_n)^{\na 2n}\na (\C P^2_{\circ})^{\na 4n(n+2)} \xrightarrow{\Pi'_{D,n}} E_{-n}^* \na (\C P^2_{\circ})^{\na (n+2)} \\
\approx P(D_{n+2},2) \na (S^2 \times S^2)_{\circ} \stackrel{K'_{D,n}}{\looparrowright}  \R^4.
\end{multline*}
$G_n'$ is a singular Seifert surface for the immersion 
\[
g_n' \co S^3 \to \pa \left(P(D_{n+2},2)\na (S^2 \times S^2)_{\circ}\right)=S^3/\dic_n \lo \R^4.
\]
The situation is summarized in the following commutative diagram; 
\[\xymatrix{
S^3 \ar@/^10mm/[rrr]^(.65){g'_n} \ar@{^{(}->}[d] \ar[r]_(.25){\text{universal covering}}&\pa \left(P(D_{n+2},2)\na (S^2 \times S^2)_{\circ}\right)=S^3/\dic_n \ar@{^{(}->}[d] \ar[rr]^(.65){\kappa'_{D,n}= \kappa_{D,n}\# t}&&\R^4\\
E(\xi_{-1})\na (X_n)^{\na 2n}\na (\C P^2_{\circ})^{\na 4n(n+2)} \ar@/^-20mm/[rrru]_(.65){G'_n} \ar[r]_ (.45){\Pi'_{D,n}}& E_{-n}^* \na (\C P^2_{\circ})^{\na (n+2)}\approx P(D_{n+2},2) \na (S^2 \times S^2)_{\circ} \ar[rru]_(.65){K'_{D,n}}
}\]
The immersion $g_n'$ is defined as a modification of the immersion $g_n$.

\textit{Step~2.}
We compute the Smale invariant of $g'_n$. The normal degree of $g'_n$ is
\begin{align*}
\ndeg(g'_n)=& \deg (S^3\to S^3/\dic_n)\cdot \ndeg(K'_{D,n})\\
=& \deg (S^3\to S^3/\dic_n)\cdot \chi\left(P(D_{n+2},2)\na (S^2 \times S^2)_{\circ}\right)\\
=& 4n\cdot (n+5).
\end{align*}
Since the signatures of $S^2\times S^2$ and $S^2 \tilde{\times} S^2$ equal $0$, the signature of $E(\xi_{-1})\na (X_n)^{\na 2n} \na (\C P^2)^{\na 4n(n+2)}$ is
\begin{align*}
\sigma\left(E(\xi_{-1})\na (X_n)^{\na 2n} \na (\C P^2)^{\na 4n(n+2)}\right) =& -1+0+4n(n+2)\\
=& 4n^2+8n-1.
\end{align*}
By \cite[Lemma~3.1]{e-t}, Lemma~\ref{lem:c} and Lemma~\ref{lem:purturb} below, the branched loci of $\Phi_1$ and $\Phi_2$ are contained in some open neighborhoods of the core of the corresponding $2$-handles $h_1$ and $h_2$ drawn in Figure~\ref{fig:d_02} (corresponding to the thickened circles). 
\begin{figure}[]
\begin{center}
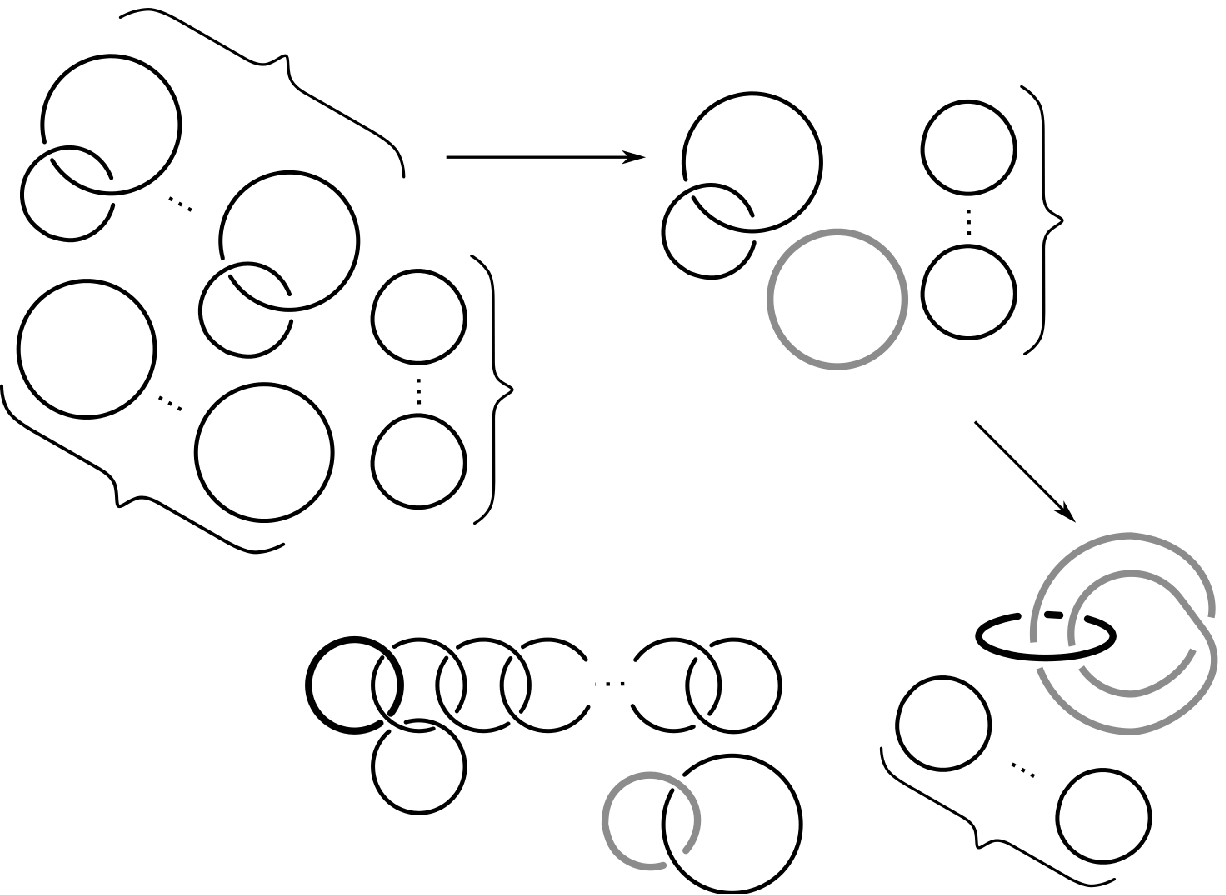
\end{center}
\caption{The branched loci of $\Phi_1$ and $\Phi_2$}\label{fig:d_02}
\end{figure}
Therefore, the branched loci of $\Phi_1$ and $\Phi_1 \circ \Phi_2$ do not overlap in $E_{-n}^* \na (\C P^2_{\circ})^{\na (n+2)} \approx P(D_{n+2},2) \na (S^2 \times S^2)_{\circ}$ (see Figure~\ref{fig:d_02}), so the algebraic number of rank $2$ points of composite $\Phi_1 \circ \Phi_2$ is given by $\deg(\Phi_2)\cdot \sharp\Sigma^{2}(\Phi_1)+\sharp\Sigma^{2}(\Phi_2)$ by Lemma~\ref{lem:2com}. Thus, the algebraic number of rank $2$ points of the map $G'_n= K'_{D,n} \circ \Pi'_{D,n}$ is
\begin{align*}
\sharp\Sigma^2(G_n')
&= \sharp\Sigma^2(K'_{D,n} \circ \Pi'_{D,n})=\sharp\Sigma^2(\Pi'_{D,n})\\
&=\sharp\Sigma^2(\Phi_1 \circ \Phi_2)=\deg (\Phi_2) \cdot \sharp\Sigma^2(\Phi_1)+\sharp\Sigma^2(\Phi_2)\\
&= 2n \cdot 0-(2n)^2+1=-4n^2+1.
\end{align*}
Thus the Hirzebruch defect of $g'_n$ is
\begin{align*}
\hdef(g'_n)
&= -3\sigma\left(E(\xi_{-1})\na (X_n)^{\na 2n} \na (\C P^2)^{\na 4n(n+2)}\right)-\sharp\Sigma^2(G_n')\\
&= -8n^2-24n+2
\end{align*}
by \eqref{eq:hirzebruch}. Substituting these values into the formula \eqref{eq:smale}, we have
\[
\Omega(g_n')=\left(4n(n+5)-1, \frac{-(-8n^2-24n+2)-2(4n^2+20n-1)}{4} \right)=(4n^2+20n-1,-4n).
\]

\textit{Step~3.}
By Lemma~\ref{lem:step3} and the proof of type A, the difference of Smale invariants of $g_n$ and $g'_n$ is
\[
\Omega(g_n)=\Omega(g'_n)-4n\cdot \Omega(t)=(4n^2+12n-1,0).
\]

This completes the proof, and there remains only the following: 

\begin{lem}\label{lem:purturb}
Let $\Phi \co D^2 \times S^2 \to D^2 \times S^2$ be a map given by $(z,w) \mapsto (z^2,w)$. Then, $\sharp\Sigma^2(\Phi)=0$.
\end{lem}
\begin{proof}
Let $\rho \co [0,1]\to [0,1]$, $t\mapsto \rho_t$ be a smooth function such that for a small positive number $c$,
\begin{enumerate}[(i)]
\item $\rho_t=1$ for $0 \le t \le c$,
\item $\rho_t=0$ for $\frac{1}{2} \le t \le 1$,
\item $\rho'<0$ for all $c<t<\frac{1}{2}$,
\end{enumerate}
as shown in Figure~\ref{fig:bump}. The function $\rho$ will be use to perturb $\Phi$ with its values near the boundary fixed.
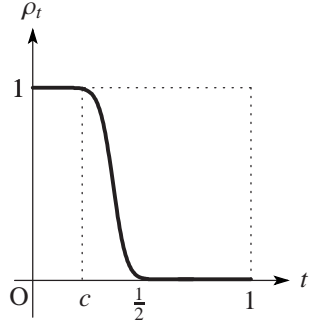
\begin{figure}[]
\begin{center}
\input{bump_fct}
\end{center}
\caption{The function $\rho$}\label{fig:bump}
\end{figure}

The first component of $\Phi$ is identity, and we consider the second component. Define $\tilde{\Phi} \co D^2 \to D^2$ by,
\[
\tilde{\Phi}(z)=z^2+\rho_{\abs{z^2}}\cdot z.
\]
If we put $z=x+iy$, then we have
\[
\tilde{\Phi} \co 
\left(\begin{array}{c} 
x \\ 
y 
\end{array} \right) 
\longmapsto 
\left(\begin{array}{c}
x^2-y^2+\rho_{x^2+y^2}\cdot x \\
2xy+\rho_{x^2+y^2}\cdot y
\end{array} \right).
\]
The Jacobian matrix $J \tilde{\Phi}$ is
\[
J \tilde{\Phi} =
\left( \begin{array}{cc}
\rho_{x^2+y^2}'\cdot 2x^2+\rho_{x^2+y^2}+2x & \rho_{x^2+y^2}'\cdot 2xy-2y\\
\rho_{x^2+y^2}' \cdot2xy+2y & \rho_{x^2+y^2}' \cdot 2y^2 + \rho_{x^2+y^2} +2x
\end{array} \right).
\]
Then, the rank of $J \tilde{\Phi}$ is not zero. Indeed, since $\rho_{\abs{z^2}}=1$ or $0$ on a small neighborhood of $0 \in D^2$ or near the boundary of $D^2$ respectively, we can check that the Jacobian $\abs{J\tilde{\Phi}}$ at any point in those areas are not zero. In addition, if we assume that the Jacobian matrix $J\tilde{\Phi}$ is zero for some $z$ with $c < \abs{z^2} < \frac{1}{2}$, then 
\[
\begin{cases}
\rho_{x^2+y^2}'\cdot 2x^2+\rho_{x^2+y^2}+2x=0 \\
\rho_{x^2+y^2}'\cdot 2xy-2y=0 \\
\rho_{x^2+y^2}' \cdot 2xy+2y=0 \\
\rho_{x^2+y^2}' \cdot 2y^2 + \rho_{x^2+y^2} +2x=0 .
\end{cases}
\]
This implies $x=y=0$. However, this is inconsistent with $c < \abs{z^2} < \frac{1}{2}$. Thus the map $\tilde{\Phi}$ has no rank $2$ points. Since $\tilde{\Phi}$ is homotopic to a map given by $z \mapsto z^2$, we have $\sharp\Sigma^2(\Phi)=0$.
\end{proof}

\begin{rem}\label{rem:d}
In Figure~\ref{fig:dic_1}, the transformation from (a) to (h) proves 
\begin{equation}\label{eq:d_02}
E_{-n}^* \na (\C P^2_{\circ})^{\na (n+1)} \approx P(D_{n+2},2)\na \overline{\C P^2}_{\circ}.
\end{equation}
By the construction of the map $\Pi'_{D,n} =\Phi_1 \circ \Phi_2$ and Figure~\ref{fig:d_02}, it seems to be possible that we can obtain a branched $\dic_n$-covering over $P(D_{n+2},2)$ by deleting one $\overline{\C P^2}$ from \eqref{eq:d_02}, but it is not a ``good'' covering. The reason for that is almost the same as the case of $A_n$. Therefore we consider the immersion $g'_n$ which is defined as a modification of the immersion $g_n$. 

Consider the connected sum of $P(D_{n+2},2)\na \overline{\C P^2}_{\circ}$ and two $\C P^2$'s. Since $\C P^2 \# \C P^2 \# \overline{\C P^2}$ is diffeomorphic to $(S^2 \times S^2) \# \C P^2$ (see \cite[Corollary~4.3]{kirby}), we obtain Figure~\ref{fig:dic_1} (i). Then deleting one $\C P^2$ in (a) and (i), we have \eqref{eq:d_01}.
\end{rem}

\section*{acknowledgments}
The author would like to thank Keiichi Sakai, Masamichi Takase and Atsuko Katanaga for their valuable advices. She also would like to thank Minoru Yamamoto, Kentaro Saji and Kokoro Tanaka for their supports and encouragements.

\end{document}

%% file: k_move1.eps_tex
\begingroup%
  \makeatletter%
  \providecommand\color[2][]{%
    \errmessage{(Inkscape) Color is used for the text in Inkscape, but the package 'color.sty' is not loaded}%
    \renewcommand\color[2][]{}%
  }%
  \providecommand\transparent[1]{%
    \errmessage{(Inkscape) Transparency is used (non-zero) for the text in Inkscape, but the package 'transparent.sty' is not loaded}%
    \renewcommand\transparent[1]{}%
  }%
  \providecommand\rotatebox[2]{#2}%
  \ifx\svgwidth\undefined%
    \setlength{\unitlength}{230.63787655bp}%
    \ifx\svgscale\undefined%
      \relax%
    \else%
      \setlength{\unitlength}{\unitlength * \real{\svgscale}}%
    \fi%
  \else%
    \setlength{\unitlength}{\svgwidth}%
  \fi%
  \global\let\svgwidth\undefined%
  \global\let\svgscale\undefined%
  \makeatother%
  \begin{picture}(1,0.24490149)%
    \put(0,0){\includegraphics[width=\unitlength]{k_move1.eps}}%
    \put(0.96350363,0.02044256){\color[rgb]{0,0,0}\makebox(0,0)[lb]{\smash{$\sb{\pm 1}$}}}%
  \end{picture}%
\endgroup%

%% file: k_move2.eps_tex
\begingroup%
  \makeatletter%
  \providecommand\color[2][]{%
    \errmessage{(Inkscape) Color is used for the text in Inkscape, but the package 'color.sty' is not loaded}%
    \renewcommand\color[2][]{}%
  }%
  \providecommand\transparent[1]{%
    \errmessage{(Inkscape) Transparency is used (non-zero) for the text in Inkscape, but the package 'transparent.sty' is not loaded}%
    \renewcommand\transparent[1]{}%
  }%
  \providecommand\rotatebox[2]{#2}%
  \ifx\svgwidth\undefined%
    \setlength{\unitlength}{350.84778183bp}%
    \ifx\svgscale\undefined%
      \relax%
    \else%
      \setlength{\unitlength}{\unitlength * \real{\svgscale}}%
    \fi%
  \else%
    \setlength{\unitlength}{\svgwidth}%
  \fi%
  \global\let\svgwidth\undefined%
  \global\let\svgscale\undefined%
  \makeatother%
  \begin{picture}(1,0.50643064)%
    \put(0,0){\includegraphics[width=\unitlength]{k_move2.eps}}%
    \put(0.10113203,0.45193655){\color[rgb]{0,0,0}\makebox(0,0)[lb]{\smash{$L_1$}}}%
    \put(0.17347017,0.39931875){\color[rgb]{0,0,0}\makebox(0,0)[lb]{\smash{$L_2$}}}%
    \put(0.0801746,0.34317833){\color[rgb]{0,0,0}\makebox(0,0)[lb]{\smash{$\sb{n_1}$}}}%
    \put(0.75489281,0.22021881){\color[rgb]{0,0,0}\makebox(0,0)[lb]{\smash{or}}}%
    \put(0.26616925,0.32662959){\color[rgb]{0,0,0}\makebox(0,0)[lb]{\smash{$\sb{n_2}$}}}%
    \put(0.69598528,0.42809809){\color[rgb]{0,0,0}\makebox(0,0)[lb]{\smash{$L'_1$}}}%
    \put(0.83393684,0.42125752){\color[rgb]{0,0,0}\makebox(0,0)[lb]{\smash{$L_2$}}}%
    \put(0.83305761,0.38021408){\color[rgb]{0,0,0}\makebox(0,0)[lb]{\smash{$\sb{n_2}$}}}%
    \put(0.70738623,0.13053315){\color[rgb]{0,0,0}\makebox(0,0)[lb]{\smash{$L'_1$}}}%
    \put(0.84191751,0.12711286){\color[rgb]{0,0,0}\makebox(0,0)[lb]{\smash{$L_2$}}}%
    \put(0.84331847,0.08834961){\color[rgb]{0,0,0}\makebox(0,0)[lb]{\smash{$\sb{n_2}$}}}%
    \put(0.6557045,0.29318493){\color[rgb]{0,0,0}\makebox(0,0)[lb]{\smash{$\sb{n'_1:=n_1+n_2+2{\text{Link}}(L_1,L_2)}$}}}%
    \put(0.66394681,0.00374428){\color[rgb]{0,0,0}\makebox(0,0)[lb]{\smash{$\sb{n'_1:=n_1+n_2-2{\text{Link}}(L_1,L_2)}$}}}%
    \put(0.87942487,0.27876817){\color[rgb]{0,0,0}\makebox(0,0)[lb]{\smash{full $n_2$-twist}}}%
    \put(0.90717796,0){\color[rgb]{0,0,0}\makebox(0,0)[lb]{\smash{full $n_2$-twist}}}%
  \end{picture}%
\endgroup%

%% file: ex_handle_slide.eps_tex
\begingroup%
  \makeatletter%
  \providecommand\color[2][]{%
    \errmessage{(Inkscape) Color is used for the text in Inkscape, but the package 'color.sty' is not loaded}%
    \renewcommand\color[2][]{}%
  }%
  \providecommand\transparent[1]{%
    \errmessage{(Inkscape) Transparency is used (non-zero) for the text in Inkscape, but the package 'transparent.sty' is not loaded}%
    \renewcommand\transparent[1]{}%
  }%
  \providecommand\rotatebox[2]{#2}%
  \ifx\svgwidth\undefined%
    \setlength{\unitlength}{302.64191895bp}%
    \ifx\svgscale\undefined%
      \relax%
    \else%
      \setlength{\unitlength}{\unitlength * \real{\svgscale}}%
    \fi%
  \else%
    \setlength{\unitlength}{\svgwidth}%
  \fi%
  \global\let\svgwidth\undefined%
  \global\let\svgscale\undefined%
  \makeatother%
  \begin{picture}(1,0.24926912)%
    \put(0,0){\includegraphics[width=\unitlength]{ex_handle_slide.eps}}%
    \put(0.00478305,0.0709696){\color[rgb]{0,0,0}\makebox(0,0)[lb]{\smash{$\sb{+1}$}}}%
    \put(0.41422461,0.07085444){\color[rgb]{0,0,0}\makebox(0,0)[lb]{\smash{$\sb{+1}$}}}%
    \put(0.54916426,0.0709696){\color[rgb]{0,0,0}\makebox(0,0)[lb]{\smash{$\sb{-1}$}}}%
    \put(0.97438185,0.07940906){\color[rgb]{0,0,0}\makebox(0,0)[lb]{\smash{$\sb{-1}$}}}%
    \put(0.12363699,0.00753303){\color[rgb]{0,0,0}\makebox(0,0)[lb]{\smash{$L$}}}%
    \put(0.66301282,0.01174113){\color[rgb]{0,0,0}\makebox(0,0)[lb]{\smash{$L$}}}%
    \put(0.17846862,0.11702899){\color[rgb]{0,0,0}\makebox(0,0)[lb]{\smash{slide}}}%
    \put(0.7350904,0.11702898){\color[rgb]{0,0,0}\makebox(0,0)[lb]{\smash{slide}}}%
  \end{picture}%
\endgroup%

%% file: example.eps_tex
\begingroup%
  \makeatletter%
  \providecommand\color[2][]{%
    \errmessage{(Inkscape) Color is used for the text in Inkscape, but the package 'color.sty' is not loaded}%
    \renewcommand\color[2][]{}%
  }%
  \providecommand\transparent[1]{%
    \errmessage{(Inkscape) Transparency is used (non-zero) for the text in Inkscape, but the package 'transparent.sty' is not loaded}%
    \renewcommand\transparent[1]{}%
  }%
  \providecommand\rotatebox[2]{#2}%
  \ifx\svgwidth\undefined%
    \setlength{\unitlength}{346.4106539bp}%
    \ifx\svgscale\undefined%
      \relax%
    \else%
      \setlength{\unitlength}{\unitlength * \real{\svgscale}}%
    \fi%
  \else%
    \setlength{\unitlength}{\svgwidth}%
  \fi%
  \global\let\svgwidth\undefined%
  \global\let\svgscale\undefined%
  \makeatother%
  \begin{picture}(1,0.40492772)%
    \put(0,0){\includegraphics[width=\unitlength]{example.eps}}%
    \put(0.57827863,0.30289708){\color[rgb]{0,0,0}\makebox(0,0)[lb]{\smash{or}}}%
    \put(0.53710879,0.3825026){\color[rgb]{0,0,0}\makebox(0,0)[lb]{\smash{$\sb{k+2}$}}}%
    \put(0.79331085,0.38717761){\color[rgb]{0,0,0}\makebox(0,0)[lb]{\smash{$\sb{k-2}$}}}%
    \put(0.19913578,0.24366506){\color[rgb]{0,0,0}\makebox(0,0)[lb]{\smash{$\sb{k}$}}}%
    \put(0.23275785,0.08695487){\color[rgb]{0,0,0}\makebox(0,0)[lb]{\smash{$=$}}}%
    \put(0.06857849,0.015){\color[rgb]{0,0,0}\makebox(0,0)[lb]{\smash{$\sb{2k}$}}}%
    \put(0.15001502,0.015){\color[rgb]{0,0,0}\makebox(0,0)[lb]{\smash{$\sb{0}$}}}%
    \put(0.32041267,0.015){\color[rgb]{0,0,0}\makebox(0,0)[lb]{\smash{$\sb{0}$}}}%
    \put(0.39431375,0.015){\color[rgb]{0,0,0}\makebox(0,0)[lb]{\smash{$\sb{0}$}}}%
    \put(0.76944668,0.0892328){\color[rgb]{0,0,0}\makebox(0,0)[lb]{\smash{$=$}}}%
    \put(0.58977473,0.015){\color[rgb]{0,0,0}\makebox(0,0)[lb]{\smash{$\sb{2k+1}$}}}%
    \put(0.68670384,0.015){\color[rgb]{0,0,0}\makebox(0,0)[lb]{\smash{$\sb{0}$}}}%
    \put(0.8571015,0.015){\color[rgb]{0,0,0}\makebox(0,0)[lb]{\smash{$\sb{1}$}}}%
    \put(0.93100258,0.015){\color[rgb]{0,0,0}\makebox(0,0)[lb]{\smash{$\sb{0}$}}}%
    \put(-0.00223272,0.38580174){\color[rgb]{0,0,0}\makebox(0,0)[lb]{\smash{(a)}}}%
    \put(0.32768124,0.38580174){\color[rgb]{0,0,0}\makebox(0,0)[lb]{\smash{(b)}}}%
    \put(-0.00223272,0.15365209){\color[rgb]{0,0,0}\makebox(0,0)[lb]{\smash{(c)}}}%
    \put(0.53354758,0.15365209){\color[rgb]{0,0,0}\makebox(0,0)[lb]{\smash{(d)}}}%
  \end{picture}%
\endgroup%

%% file: covering_1.eps_tex
\begingroup%
  \makeatletter%
  \providecommand\color[2][]{%
    \errmessage{(Inkscape) Color is used for the text in Inkscape, but the package 'color.sty' is not loaded}%
    \renewcommand\color[2][]{}%
  }%
  \providecommand\transparent[1]{%
    \errmessage{(Inkscape) Transparency is used (non-zero) for the text in Inkscape, but the package 'transparent.sty' is not loaded}%
    \renewcommand\transparent[1]{}%
  }%
  \providecommand\rotatebox[2]{#2}%
  \ifx\svgwidth\undefined%
    \setlength{\unitlength}{337.1953784bp}%
    \ifx\svgscale\undefined%
      \relax%
    \else%
      \setlength{\unitlength}{\unitlength * \real{\svgscale}}%
    \fi%
  \else%
    \setlength{\unitlength}{\svgwidth}%
  \fi%
  \global\let\svgwidth\undefined%
  \global\let\svgscale\undefined%
  \makeatother%
  \begin{picture}(1,0.29599017)%
    \put(0,0){\includegraphics[width=\unitlength]{covering_1.eps}}%
    \put(0.25024277,0.13447679){\color[rgb]{0,0,0}\makebox(0,0)[lb]{\smash{diagram}}}%
    \put(0.89336146,0.13396877){\color[rgb]{0,0,0}\makebox(0,0)[lb]{\smash{diagram}}}%
    \put(0.84416085,0.1085854){\color[rgb]{0,0,0}\rotatebox{-90}{\makebox(0,0)[lb]{\smash{diagram}}}}%
    \put(0.85696517,0.19435163){\color[rgb]{0,0,0}\rotatebox{90}{\makebox(0,0)[lb]{\smash{diagram}}}}%
    \put(0.48055693,0.16347979){\color[rgb]{0,0,0}\makebox(0,0)[lb]{\smash{$3:1$}}}%
    \put(0.30770249,0.24143375){\color[rgb]{0,0,0}\makebox(0,0)[lb]{\smash{$Y$}}}%
    \put(0.94997537,0.26007492){\color[rgb]{0,0,0}\makebox(0,0)[lb]{\smash{$X$}}}%
  \end{picture}%
\endgroup%

%% file: covering_dic_1.eps_tex
\begingroup%
  \makeatletter%
  \providecommand\color[2][]{%
    \errmessage{(Inkscape) Color is used for the text in Inkscape, but the package 'color.sty' is not loaded}%
    \renewcommand\color[2][]{}%
  }%
  \providecommand\transparent[1]{%
    \errmessage{(Inkscape) Transparency is used (non-zero) for the text in Inkscape, but the package 'transparent.sty' is not loaded}%
    \renewcommand\transparent[1]{}%
  }%
  \providecommand\rotatebox[2]{#2}%
  \ifx\svgwidth\undefined%
    \setlength{\unitlength}{235.98892822bp}%
    \ifx\svgscale\undefined%
      \relax%
    \else%
      \setlength{\unitlength}{\unitlength * \real{\svgscale}}%
    \fi%
  \else%
    \setlength{\unitlength}{\svgwidth}%
  \fi%
  \global\let\svgwidth\undefined%
  \global\let\svgscale\undefined%
  \makeatother%
  \begin{picture}(1,0.3071628)%
    \put(0,0){\includegraphics[width=\unitlength]{covering_dic_1.eps}}%
    \put(0.66701898,0.25744332){\color[rgb]{0,0,0}\makebox(0,0)[lb]{\smash{$X$}}}%
    \put(0.30687625,0.26505176){\color[rgb]{0,0,0}\makebox(0,0)[lb]{\smash{$\sb{-k-2}$}}}%
    \put(0.44586912,0.1536385){\color[rgb]{0,0,0}\makebox(0,0)[lb]{\smash{$2:1$}}}%
    \put(0.93728881,0.22812117){\color[rgb]{0,0,0}\makebox(0,0)[lb]{\smash{$\sb{-k-2}$}}}%
    \put(0.9492743,0.15620856){\color[rgb]{0,0,0}\makebox(0,0)[lb]{\smash{$\sb{-k-2}$}}}%
    \put(-0.01032774,0.28230894){\color[rgb]{0,0,0}\makebox(0,0)[lb]{\smash{$Y=E(\eta_{-k})$}}}%
  \end{picture}%
\endgroup%

%% file: covering_dic_2.eps_tex
\begingroup%
  \makeatletter%
  \providecommand\color[2][]{%
    \errmessage{(Inkscape) Color is used for the text in Inkscape, but the package 'color.sty' is not loaded}%
    \renewcommand\color[2][]{}%
  }%
  \providecommand\transparent[1]{%
    \errmessage{(Inkscape) Transparency is used (non-zero) for the text in Inkscape, but the package 'transparent.sty' is not loaded}%
    \renewcommand\transparent[1]{}%
  }%
  \providecommand\rotatebox[2]{#2}%
  \ifx\svgwidth\undefined%
    \setlength{\unitlength}{307.2912926bp}%
    \ifx\svgscale\undefined%
      \relax%
    \else%
      \setlength{\unitlength}{\unitlength * \real{\svgscale}}%
    \fi%
  \else%
    \setlength{\unitlength}{\svgwidth}%
  \fi%
  \global\let\svgwidth\undefined%
  \global\let\svgscale\undefined%
  \makeatother%
  \begin{picture}(1,0.52981544)%
    \put(0,0){\includegraphics[width=\unitlength]{covering_dic_2.eps}}%
    \put(0.2430854,0.46806069){\color[rgb]{0,0,0}\makebox(0,0)[lb]{\smash{$\sb{-k-2}$}}}%
    \put(0.25228984,0.41283433){\color[rgb]{0,0,0}\makebox(0,0)[lb]{\smash{$\sb{-k-2}$}}}%
    \put(0.19596062,0.05810678){\color[rgb]{0,0,0}\makebox(0,0)[lb]{\smash{$\sb{-2k}$}}}%
    \put(0.80314829,0.17262322){\color[rgb]{0,0,0}\makebox(0,0)[lb]{\smash{$\sb{\text{full}~(-k-2)-\text{twist}}$}}}%
    \put(0.86068693,0.31486292){\color[rgb]{0,0,0}\makebox(0,0)[lb]{\smash{$\sb{-k-2}$}}}%
    \put(0.93419723,0.37699251){\color[rgb]{0,0,0}\makebox(0,0)[lb]{\smash{$\sb{-2k-2}$}}}%
    \put(0.54401964,0.20314102){\color[rgb]{0,0,0}\makebox(0,0)[lb]{\smash{$\sb{-2k-2}$}}}%
    \put(0.60350625,0.00625043){\color[rgb]{0,0,0}\makebox(0,0)[lb]{\smash{cancelling pair}}}%
    \put(0.66600674,0.07935721){\color[rgb]{0,0,0}\makebox(0,0)[lb]{\smash{$\sb{-k-2}$}}}%
    \put(0.00761048,0.5025041){\color[rgb]{0,0,0}\makebox(0,0)[lb]{\smash{(a)~$X$}}}%
    \put(0.5374195,0.30523012){\color[rgb]{0,0,0}\makebox(0,0)[lb]{\smash{$\sb{\text{full}~(-k-2)-\text{twist}}$}}}%
    \put(0.5904808,0.44746984){\color[rgb]{0,0,0}\makebox(0,0)[lb]{\smash{$\sb{-k-2}$}}}%
    \put(0.66399103,0.50959943){\color[rgb]{0,0,0}\makebox(0,0)[lb]{\smash{$\sb{-2k-2}$}}}%
    \put(0.4053842,0.5096099){\color[rgb]{0,0,0}\makebox(0,0)[lb]{\smash{(b)}}}%
    \put(0.96607708,0.33538861){\color[rgb]{0,0,0}\makebox(0,0)[lb]{\smash{(c)}}}%
    \put(0.34148577,0.2052432){\color[rgb]{0,0,0}\makebox(0,0)[lb]{\smash{(d)}}}%
    \put(0.02984851,0.20391441){\color[rgb]{0,0,0}\makebox(0,0)[lb]{\smash{(e)~$E(\xi_{-2k})$}}}%
    \put(0.32589827,0.43159346){\color[rgb]{0,0,0}\makebox(0,0)[lb]{\smash{slide}}}%
    \put(0.65587916,0.15293108){\color[rgb]{0,0,0}\rotatebox{19.60654257}{\makebox(0,0)[lb]{\smash{isotopy}}}}%
    \put(0.24671003,0.08479684){\color[rgb]{0,0,0}\makebox(0,0)[lb]{\smash{isotopy}}}%
    \put(0.73899422,0.41301071){\color[rgb]{0,0,0}\rotatebox{-30.00000011}{\makebox(0,0)[lb]{\smash{$=$}}}}%
  \end{picture}%
\endgroup%

%% file: D_5.tex
\unitlength 0.1in
\begin{picture}( 34.7000,  9.4600)(  9.8000,-17.7600)
%
{\color[named]{Black}{%
\special{pn 8}%
\special{ar 3710 1210 210 210  4.0977260  0.6435011}%
}}%
%
{\color[named]{Black}{%
\special{pn 8}%
\special{ar 3960 1210 210 210  4.2895351  0.6435011}%
}}%
%
{\color[named]{Black}{%
\special{pn 8}%
\special{ar 3960 1210 210 210  0.9928944  3.8899707}%
}}%
%
{\color[named]{Black}{%
\special{pn 8}%
\special{ar 4240 1210 210 210  4.1590947  3.6901421}%
}}%
%
{\color[named]{Black}{%
\special{pn 8}%
\special{ar 3700 1570 206 206  4.3484100  3.9614599}%
}}%
%
{\color[named]{Black}{%
\special{pn 8}%
\special{ar 3400 1210 206 206  0.9685090  0.5070985}%
}}%
%
{\color[named]{Black}{%
\special{pn 8}%
\special{ar 3710 1210 210 210  1.1760052  3.6052403}%
}}%
\put(34.3000,-9.9000){\makebox(0,0)[rb]{$\sb{m_1}$}}%
\put(35.2000,-17.0000){\makebox(0,0)[rt]{$\sb{m_2}$}}%
\put(36.8000,-9.7000){\makebox(0,0)[lb]{$\sb{m_3}$}}%
\put(39.7000,-9.6000){\makebox(0,0)[lb]{$\sb{m_4}$}}%
\put(42.2000,-9.7000){\makebox(0,0)[lb]{$\sb{m_5}$}}%
\put(9.8000,-11.5000){\makebox(0,0)[lt]{\xygraph{\bullet ([]!{+(0,-.3)} {m_1}) - [r] \bullet ([]!{+(.3,-.3)} {m_3}) (- [d] \bullet ([]!{+(.3,0)} {m_2}), - [r] \bullet ([]!{+(0,-.3)} {m_4}) - [r] \bullet ([]!{+(0,-.3)} {m_5}))}}}%
\put(29.2000,-13.2000){\makebox(0,0){$=$}}%
\put(10.7000,-10.3000){\makebox(0,0)[lb]{${P(D_5,m_1,\dots,m_5)}$}}%
\end{picture}%

%% file: a_01.eps_tex
\begingroup%
  \makeatletter%
  \providecommand\color[2][]{%
    \errmessage{(Inkscape) Color is used for the text in Inkscape, but the package 'color.sty' is not loaded}%
    \renewcommand\color[2][]{}%
  }%
  \providecommand\transparent[1]{%
    \errmessage{(Inkscape) Transparency is used (non-zero) for the text in Inkscape, but the package 'transparent.sty' is not loaded}%
    \renewcommand\transparent[1]{}%
  }%
  \providecommand\rotatebox[2]{#2}%
  \ifx\svgwidth\undefined%
    \setlength{\unitlength}{359.00469462bp}%
    \ifx\svgscale\undefined%
      \relax%
    \else%
      \setlength{\unitlength}{\unitlength * \real{\svgscale}}%
    \fi%
  \else%
    \setlength{\unitlength}{\svgwidth}%
  \fi%
  \global\let\svgwidth\undefined%
  \global\let\svgscale\undefined%
  \makeatother%
  \begin{picture}(1,1.37841772)%
    \put(0,0){\includegraphics[width=\unitlength]{a_01.eps}}%
    \put(0.38551426,1.25826529){\color[rgb]{0,0,0}\makebox(0,0)[lb]{\smash{$\sb{1}$}}}%
    \put(0.51205457,1.17072169){\color[rgb]{0,0,0}\makebox(0,0)[lb]{\smash{$\sb{1}$}}}%
    \put(0.51682967,1.31477073){\color[rgb]{0,0,0}\makebox(0,0)[lb]{\smash{$\sb{1}$}}}%
    \put(0.52956324,1.02587679){\color[rgb]{0,0,0}\makebox(0,0)[lb]{\smash{$\sb{1}$}}}%
    \put(0.61233181,1.02667265){\color[rgb]{0,0,0}\makebox(0,0)[lb]{\smash{$\sb{1}$}}}%
    \put(0.7218809,1.13531931){\color[rgb]{0,0,0}\makebox(0,0)[lb]{\smash{$\sb{-n+1}$}}}%
    \put(0.45045979,1.14300345){\color[rgb]{0,0,0}\makebox(0,0)[lb]{\smash{$\sb{n-2}$}}}%
    \put(-0.00215439,1.36112273){\color[rgb]{0,0,0}\makebox(0,0)[lb]{\smash{(a)~$E(\xi_{-n})\na(\C P^2_{\circ})^{\na n}\na \C P^2_{\circ}$}}}%
    \put(0.39429222,1.31841068){\color[rgb]{0,0,0}\makebox(0,0)[lb]{\smash{(b)}}}%
    \put(0.00528014,1.21245816){\color[rgb]{0,0,0}\makebox(0,0)[lb]{\smash{1$\sb{1}$}}}%
    \put(0.13573011,1.13009044){\color[rgb]{0,0,0}\makebox(0,0)[lb]{\smash{$\sb{1}$}}}%
    \put(0.14607616,0.97887872){\color[rgb]{0,0,0}\makebox(0,0)[lb]{\smash{$\sb{1}$}}}%
    \put(0.23202811,0.97808286){\color[rgb]{0,0,0}\makebox(0,0)[lb]{\smash{$\sb{1}$}}}%
    \put(0.24642308,1.20768303){\color[rgb]{0,0,0}\makebox(0,0)[lb]{\smash{$\sb{-n}$}}}%
    \put(0.17817938,0.89078893){\color[rgb]{0,0,0}\makebox(0,0)[lb]{\smash{(d)}}}%
    \put(0.34389973,0.63290976){\color[rgb]{0,0,0}\makebox(0,0)[lb]{\smash{$\sb{1}$}}}%
    \put(0.43144333,0.63211383){\color[rgb]{0,0,0}\makebox(0,0)[lb]{\smash{$\sb{1}$}}}%
    \put(0.20701333,0.85574807){\color[rgb]{0,0,0}\makebox(0,0)[lb]{\smash{$\sb{1}$}}}%
    \put(0.33912463,0.76581682){\color[rgb]{0,0,0}\makebox(0,0)[lb]{\smash{$\sb{1}$}}}%
    \put(0.31781026,0.98575084){\color[rgb]{0,0,0}\makebox(0,0)[lb]{\smash{$\sb{-n+2}$}}}%
    \put(0.48954049,0.91464102){\color[rgb]{0,0,0}\makebox(0,0)[lb]{\smash{$\sb{1}$}}}%
    \put(0.53729153,0.88599037){\color[rgb]{0,0,0}\makebox(0,0)[lb]{\smash{$\sb{2}$}}}%
    \put(0.28884878,0.22731697){\color[rgb]{0,0,0}\makebox(0,0)[lb]{\smash{(f)~$P(A_{n-1},2)\na \overline{\C P^2}_{\circ}\na (\C P^2_{\circ})^{\na 2}$}}}%
    \put(0.39027312,0.00118818){\color[rgb]{0,0,0}\makebox(0,0)[lb]{\smash{$\sb{1}$}}}%
    \put(0.47224584,0.00118818){\color[rgb]{0,0,0}\makebox(0,0)[lb]{\smash{$\sb{1}$}}}%
    \put(0.42530671,0.21094444){\color[rgb]{0,0,0}\makebox(0,0)[lb]{\smash{$\sb{2}$}}}%
    \put(0.47708085,0.21094444){\color[rgb]{0,0,0}\makebox(0,0)[lb]{\smash{$\sb{2}$}}}%
    \put(0.52638539,0.21094444){\color[rgb]{0,0,0}\makebox(0,0)[lb]{\smash{$\sb{2}$}}}%
    \put(0.63127424,0.21094444){\color[rgb]{0,0,0}\makebox(0,0)[lb]{\smash{$\sb{2}$}}}%
    \put(0.67989105,0.21094444){\color[rgb]{0,0,0}\makebox(0,0)[lb]{\smash{$\sb{2}$}}}%
    \put(0.32520995,0.006759){\color[rgb]{0,0,0}\makebox(0,0)[lb]{\smash{$\sb{-1}$}}}%
    \put(-0.00068269,0.52049433){\color[rgb]{0,0,0}\makebox(0,0)[lb]{\smash{(e)}}}%
    \put(0.05056029,0.27748659){\color[rgb]{0,0,0}\makebox(0,0)[lb]{\smash{$\sb{1}$}}}%
    \put(0.13495524,0.27957761){\color[rgb]{0,0,0}\makebox(0,0)[lb]{\smash{$\sb{1}$}}}%
    \put(0.16117196,0.50046789){\color[rgb]{0,0,0}\makebox(0,0)[lb]{\smash{$\sb{1}$}}}%
    \put(0.21294613,0.50046789){\color[rgb]{0,0,0}\makebox(0,0)[lb]{\smash{$\sb{2}$}}}%
    \put(0.26225064,0.50046789){\color[rgb]{0,0,0}\makebox(0,0)[lb]{\smash{$\sb{2}$}}}%
    \put(0.36713949,0.50046789){\color[rgb]{0,0,0}\makebox(0,0)[lb]{\smash{$\sb{2}$}}}%
    \put(0.4157563,0.50046789){\color[rgb]{0,0,0}\makebox(0,0)[lb]{\smash{$\sb{2}$}}}%
    \put(0.08005645,0.52944719){\color[rgb]{0,0,0}\makebox(0,0)[lb]{\smash{$\sb{-1}$}}}%
    \put(0.06870622,1.09455621){\color[rgb]{0,0,0}\makebox(0,0)[lb]{\smash{$\sb{n-1}$}}}%
    \put(0.80760403,1.04537149){\color[rgb]{0,0,0}\makebox(0,0)[lb]{\smash{(c)}}}%
    \put(0.69060235,0.91107136){\color[rgb]{0,0,0}\makebox(0,0)[lb]{\smash{$\sb{1}$}}}%
    \put(0.82191769,0.8267112){\color[rgb]{0,0,0}\makebox(0,0)[lb]{\smash{$\sb{1}$}}}%
    \put(0.84977261,0.68504967){\color[rgb]{0,0,0}\makebox(0,0)[lb]{\smash{$\sb{1}$}}}%
    \put(0.93731622,0.68504967){\color[rgb]{0,0,0}\makebox(0,0)[lb]{\smash{$\sb{1}$}}}%
    \put(0.98029216,0.95723074){\color[rgb]{0,0,0}\makebox(0,0)[lb]{\smash{$\sb{1}$}}}%
    \put(0.92918542,1.01879044){\color[rgb]{0,0,0}\makebox(0,0)[lb]{\smash{$\sb{-n+1}$}}}%
    \put(0.75548862,0.79736606){\color[rgb]{0,0,0}\makebox(0,0)[lb]{\smash{$\sb{n-2}$}}}%
    \put(0.26927121,0.73771433){\color[rgb]{0,0,0}\makebox(0,0)[lb]{\smash{$\sb{n-3}$}}}%
    \put(0.54180169,0.07438252){\color[rgb]{0,0,0}\makebox(0,0)[lb]{\smash{$\sb{n-1}$}}}%
    \put(0.30866384,0.35954524){\color[rgb]{0,0,0}\makebox(0,0)[lb]{\smash{$\sb{n-2}$}}}%
    \put(0.29348264,1.1721675){\color[rgb]{0,0,0}\makebox(0,0)[lb]{\smash{slide}}}%
    \put(0.70249052,1.07544692){\color[rgb]{0,0,0}\rotatebox{-41.54422829}{\makebox(0,0)[lb]{\smash{isotopy}}}}%
    \put(0.6081728,0.84310077){\color[rgb]{0,0,0}\makebox(0,0)[lb]{\smash{slide}}}%
    \put(0.18846192,0.27602745){\color[rgb]{0,0,0}\rotatebox{-45}{\makebox(0,0)[lb]{\smash{slide}}}}%
    \put(0.75885873,0.24152075){\color[rgb]{0,0,0}\rotatebox{59.99999989}{\makebox(0,0)[lb]{\smash{slide}}}}%
    \put(0.75265115,0.32901174){\color[rgb]{0,0,0}\makebox(0,0)[lb]{\smash{$\sb{1}$}}}%
    \put(0.69343831,0.54393939){\color[rgb]{0,0,0}\makebox(0,0)[lb]{\smash{$\sb{2}$}}}%
    \put(0.74521258,0.54393939){\color[rgb]{0,0,0}\makebox(0,0)[lb]{\smash{$\sb{2}$}}}%
    \put(0.79451698,0.54393939){\color[rgb]{0,0,0}\makebox(0,0)[lb]{\smash{$\sb{2}$}}}%
    \put(0.89940584,0.54393939){\color[rgb]{0,0,0}\makebox(0,0)[lb]{\smash{$\sb{2}$}}}%
    \put(0.94802264,0.54393939){\color[rgb]{0,0,0}\makebox(0,0)[lb]{\smash{$\sb{2}$}}}%
    \put(0.52502959,0.32238385){\color[rgb]{0,0,0}\makebox(0,0)[lb]{\smash{$\sb{0}$}}}%
    \put(0.66758327,0.32503128){\color[rgb]{0,0,0}\makebox(0,0)[lb]{\smash{$\sb{0}$}}}%
    \put(0.81511599,0.40569089){\color[rgb]{0,0,0}\makebox(0,0)[lb]{\smash{$\sb{n-1}$}}}%
    \put(0.5552117,0.57433423){\color[rgb]{0,0,0}\makebox(0,0)[lb]{\smash{(g)~$P(A_{n-1},2) \na (S^2 \times S^2)_{\circ}\na \C P^2_{\circ}$}}}%
    \put(0.57025934,0.81163397){\color[rgb]{0,0,0}\makebox(0,0)[lb]{\smash{Example 3.3}}}%
  \end{picture}%
\endgroup%

%% file: dic_1.eps_tex
\begingroup%
  \makeatletter%
  \providecommand\color[2][]{%
    \errmessage{(Inkscape) Color is used for the text in Inkscape, but the package 'color.sty' is not loaded}%
    \renewcommand\color[2][]{}%
  }%
  \providecommand\transparent[1]{%
    \errmessage{(Inkscape) Transparency is used (non-zero) for the text in Inkscape, but the package 'transparent.sty' is not loaded}%
    \renewcommand\transparent[1]{}%
  }%
  \providecommand\rotatebox[2]{#2}%
  \ifx\svgwidth\undefined%
    \setlength{\unitlength}{358.96440683bp}%
    \ifx\svgscale\undefined%
      \relax%
    \else%
      \setlength{\unitlength}{\unitlength * \real{\svgscale}}%
    \fi%
  \else%
    \setlength{\unitlength}{\svgwidth}%
  \fi%
  \global\let\svgwidth\undefined%
  \global\let\svgscale\undefined%
  \makeatother%
  \begin{picture}(1,1.24811178)%
    \put(0,0){\includegraphics[width=\unitlength]{dic_1.eps}}%
    \put(0.10319934,0.72816878){\color[rgb]{0,0,0}\makebox(0,0)[lb]{\smash{$\sb{1}$}}}%
    \put(0.23452942,0.64379915){\color[rgb]{0,0,0}\makebox(0,0)[lb]{\smash{$\sb{1}$}}}%
    \put(-0.00089693,1.23081485){\color[rgb]{0,0,0}\makebox(0,0)[lb]{\smash{(a)~$E_{-n}^*\na (\C P^2_{\circ})^{\na (n+2)}\na \C P^2_{\circ}$}}}%
    \put(0.40645936,1.17081935){\color[rgb]{0,0,0}\makebox(0,0)[lb]{\smash{(b)}}}%
    \put(0.77418384,1.17304798){\color[rgb]{0,0,0}\makebox(0,0)[lb]{\smash{(c)}}}%
    \put(0.53794872,0.83875301){\color[rgb]{0,0,0}\makebox(0,0)[lb]{\smash{(d)}}}%
    \put(0.11673702,0.84766755){\color[rgb]{0,0,0}\makebox(0,0)[lb]{\smash{(e)}}}%
    \put(0.01199123,0.52451573){\color[rgb]{0,0,0}\makebox(0,0)[lb]{\smash{(f)}}}%
    \put(0.16421792,0.6153494){\color[rgb]{0,0,0}\makebox(0,0)[lb]{\smash{$\sb{n+1}$}}}%
    \put(0.00882528,1.04354465){\color[rgb]{0,0,0}\makebox(0,0)[lb]{\smash{$\sb{1}$}}}%
    \put(0.14015535,0.95917502){\color[rgb]{0,0,0}\makebox(0,0)[lb]{\smash{$\sb{1}$}}}%
    \put(0.0656486,0.92838443){\color[rgb]{0,0,0}\makebox(0,0)[lb]{\smash{$\sb{n+3}$}}}%
    \put(0.07152489,1.11586874){\color[rgb]{0,0,0}\makebox(0,0)[lb]{\smash{$\sb{0}$}}}%
    \put(0.27210185,1.1804991){\color[rgb]{0,0,0}\makebox(0,0)[lb]{\smash{$\sb{-n-2}$}}}%
    \put(0.36037819,1.04223916){\color[rgb]{0,0,0}\makebox(0,0)[lb]{\smash{$\sb{1}$}}}%
    \put(0.49170826,0.95786953){\color[rgb]{0,0,0}\makebox(0,0)[lb]{\smash{$\sb{1}$}}}%
    \put(0.42051075,0.92847237){\color[rgb]{0,0,0}\makebox(0,0)[lb]{\smash{$\sb{n+3}$}}}%
    \put(0.63406996,1.1599711){\color[rgb]{0,0,0}\makebox(0,0)[lb]{\smash{$\sb{-n-2}$}}}%
    \put(0.45773316,1.1500206){\color[rgb]{0,0,0}\makebox(0,0)[lb]{\smash{$\sb{0}$}}}%
    \put(0.71049431,1.03116159){\color[rgb]{0,0,0}\makebox(0,0)[lb]{\smash{$\sb{1}$}}}%
    \put(0.84182438,0.94679197){\color[rgb]{0,0,0}\makebox(0,0)[lb]{\smash{$\sb{1}$}}}%
    \put(0.7659489,0.91621488){\color[rgb]{0,0,0}\makebox(0,0)[lb]{\smash{$\sb{n+1}$}}}%
    \put(0.82100038,1.14110607){\color[rgb]{0,0,0}\makebox(0,0)[lb]{\smash{$\sb{2}$}}}%
    \put(0.85572048,1.1198557){\color[rgb]{0,0,0}\makebox(0,0)[lb]{\smash{$\sb{1}$}}}%
    \put(0.85545605,1.07862599){\color[rgb]{0,0,0}\makebox(0,0)[lb]{\smash{$\sb{1}$}}}%
    \put(0.98285106,1.14437067){\color[rgb]{0,0,0}\makebox(0,0)[lb]{\smash{$\sb{-n}$}}}%
    \put(0.49532152,0.7207593){\color[rgb]{0,0,0}\makebox(0,0)[lb]{\smash{$\sb{1}$}}}%
    \put(0.62665159,0.63638967){\color[rgb]{0,0,0}\makebox(0,0)[lb]{\smash{$\sb{1}$}}}%
    \put(0.55088578,0.60977782){\color[rgb]{0,0,0}\makebox(0,0)[lb]{\smash{$\sb{n+1}$}}}%
    \put(0.78402062,0.83882212){\color[rgb]{0,0,0}\makebox(0,0)[lb]{\smash{$\sb{-n}$}}}%
    \put(0.60277061,0.81899022){\color[rgb]{0,0,0}\makebox(0,0)[lb]{\smash{$\sb{2}$}}}%
    \put(0.69860183,0.73095921){\color[rgb]{0,0,0}\makebox(0,0)[lb]{\smash{$\sb{2}$}}}%
    \put(0.6515361,0.78221777){\color[rgb]{0,0,0}\makebox(0,0)[lb]{\smash{$\sb{1}$}}}%
    \put(0.21164548,0.87690583){\color[rgb]{0,0,0}\makebox(0,0)[lb]{\smash{$\sb{2}$}}}%
    \put(0.30413375,0.71981604){\color[rgb]{0,0,0}\makebox(0,0)[lb]{\smash{$\sb{2}$}}}%
    \put(0.26709687,0.87802015){\color[rgb]{0,0,0}\makebox(0,0)[lb]{\smash{$\sb{1}$}}}%
    \put(0.39289549,0.86110845){\color[rgb]{0,0,0}\makebox(0,0)[lb]{\smash{$\sb{-n}$}}}%
    \put(-0.0017433,0.40612567){\color[rgb]{0,0,0}\makebox(0,0)[lb]{\smash{$\sb{1}$}}}%
    \put(0.12958678,0.32175604){\color[rgb]{0,0,0}\makebox(0,0)[lb]{\smash{$\sb{1}$}}}%
    \put(0.06727235,0.29554056){\color[rgb]{0,0,0}\makebox(0,0)[lb]{\smash{$\sb{n}$}}}%
    \put(0.10244243,0.56155425){\color[rgb]{0,0,0}\makebox(0,0)[lb]{\smash{$\sb{2}$}}}%
    \put(0.19715934,0.40780741){\color[rgb]{0,0,0}\makebox(0,0)[lb]{\smash{$\sb{2}$}}}%
    \put(0.15370099,0.56155425){\color[rgb]{0,0,0}\makebox(0,0)[lb]{\smash{$\sb{2}$}}}%
    \put(0.20580944,0.56155425){\color[rgb]{0,0,0}\makebox(0,0)[lb]{\smash{$\sb{1}$}}}%
    \put(0.33606532,0.54241392){\color[rgb]{0,0,0}\makebox(0,0)[lb]{\smash{$\sb{-n+1}$}}}%
    \put(0.72420788,0.32654081){\color[rgb]{0,0,0}\makebox(0,0)[lb]{\smash{$\sb{1}$}}}%
    \put(0.8117613,0.32654081){\color[rgb]{0,0,0}\makebox(0,0)[lb]{\smash{$\sb{1}$}}}%
    \put(0.90901616,0.56128817){\color[rgb]{0,0,0}\makebox(0,0)[lb]{\smash{(g)}}}%
    \put(0.63917165,0.54483948){\color[rgb]{0,0,0}\makebox(0,0)[lb]{\smash{$\sb{2}$}}}%
    \put(0.69191595,0.54483948){\color[rgb]{0,0,0}\makebox(0,0)[lb]{\smash{$\sb{2}$}}}%
    \put(0.53368299,0.54483948){\color[rgb]{0,0,0}\makebox(0,0)[lb]{\smash{$\sb{2}$}}}%
    \put(0.58642729,0.54483948){\color[rgb]{0,0,0}\makebox(0,0)[lb]{\smash{$\sb{2}$}}}%
    \put(0.78859706,0.54483948){\color[rgb]{0,0,0}\makebox(0,0)[lb]{\smash{$\sb{1}$}}}%
    \put(0.93122959,0.5092102){\color[rgb]{0,0,0}\makebox(0,0)[lb]{\smash{$\sb{-1}$}}}%
    \put(0.58939882,0.36546335){\color[rgb]{0,0,0}\makebox(0,0)[lb]{\smash{$\sb{2}$}}}%
    \put(0.30917624,1.12817963){\color[rgb]{0,0,0}\makebox(0,0)[lb]{\smash{isotopy}}}%
    \put(0.85820513,0.82124283){\color[rgb]{0,0,0}\rotatebox{45.00000032}{\makebox(0,0)[lb]{\smash{slide}}}}%
    \put(0.40159,0.73481551){\color[rgb]{0,0,0}\makebox(0,0)[lb]{\smash{isotopy}}}%
    \put(0.06284327,0.63807849){\color[rgb]{0,0,0}\makebox(0,0)[lb]{\smash{slide}}}%
    \put(0.4667011,0.2622368){\color[rgb]{0,0,0}\rotatebox{30.00000043}{\makebox(0,0)[lb]{\smash{slide}}}}%
    \put(0.58274433,0.24137865){\color[rgb]{0,0,0}\makebox(0,0)[lb]{\smash{(i)~$P(D_{n+2},2) \na (S^2 \times S^2)_{\circ}\na \C P^2_{\circ}$}}}%
    \put(0.76498965,0.0061209){\color[rgb]{0,0,0}\makebox(0,0)[lb]{\smash{$\sb{1}$}}}%
    \put(0.84352802,0.00683476){\color[rgb]{0,0,0}\makebox(0,0)[lb]{\smash{$\sb{0}$}}}%
    \put(0.96487087,0.00683476){\color[rgb]{0,0,0}\makebox(0,0)[lb]{\smash{$\sb{0}$}}}%
    \put(0.71774498,0.22514096){\color[rgb]{0,0,0}\makebox(0,0)[lb]{\smash{$\sb{2}$}}}%
    \put(0.77048927,0.22514096){\color[rgb]{0,0,0}\makebox(0,0)[lb]{\smash{$\sb{2}$}}}%
    \put(0.61225632,0.22514096){\color[rgb]{0,0,0}\makebox(0,0)[lb]{\smash{$\sb{2}$}}}%
    \put(0.66500061,0.22514096){\color[rgb]{0,0,0}\makebox(0,0)[lb]{\smash{$\sb{2}$}}}%
    \put(0.8682847,0.22514096){\color[rgb]{0,0,0}\makebox(0,0)[lb]{\smash{$\sb{2}$}}}%
    \put(0.91914281,0.22514096){\color[rgb]{0,0,0}\makebox(0,0)[lb]{\smash{$\sb{2}$}}}%
    \put(0.65765751,0.04930529){\color[rgb]{0,0,0}\makebox(0,0)[lb]{\smash{$\sb{2}$}}}%
    \put(0.49528047,0.15091715){\color[rgb]{0,0,0}\rotatebox{-1.01988823}{\makebox(0,0)[lb]{\smash{slide}}}}%
    \put(0.17405219,0.22140911){\color[rgb]{0,0,0}\makebox(0,0)[lb]{\smash{$\sb{2}$}}}%
    \put(0.22679649,0.22140911){\color[rgb]{0,0,0}\makebox(0,0)[lb]{\smash{$\sb{2}$}}}%
    \put(0.06856353,0.22140911){\color[rgb]{0,0,0}\makebox(0,0)[lb]{\smash{$\sb{2}$}}}%
    \put(0.12130783,0.22140911){\color[rgb]{0,0,0}\makebox(0,0)[lb]{\smash{$\sb{2}$}}}%
    \put(0.32459192,0.22140911){\color[rgb]{0,0,0}\makebox(0,0)[lb]{\smash{$\sb{2}$}}}%
    \put(0.37545003,0.22140911){\color[rgb]{0,0,0}\makebox(0,0)[lb]{\smash{$\sb{2}$}}}%
    \put(0.11396469,0.04557338){\color[rgb]{0,0,0}\makebox(0,0)[lb]{\smash{$\sb{2}$}}}%
    \put(0.19576265,0){\color[rgb]{0,0,0}\makebox(0,0)[lb]{\smash{$\sb{1}$}}}%
    \put(0.28331608,0){\color[rgb]{0,0,0}\makebox(0,0)[lb]{\smash{$\sb{1}$}}}%
    \put(0.4369206,0.00991518){\color[rgb]{0,0,0}\makebox(0,0)[lb]{\smash{$\sb{-1}$}}}%
    \put(0.01010326,0.23734703){\color[rgb]{0,0,0}\makebox(0,0)[lb]{\smash{(h)~$P(D_{n+2},2) \na \overline{\C P^2}_{\circ}\na (\C P^2_{\circ})^{\na 2}$}}}%
    \put(0.65998033,1.11436912){\color[rgb]{0,0,0}\makebox(0,0)[lb]{\smash{Example 3.3}}}%
    \put(0.68359549,1.08338737){\color[rgb]{0,0,0}\makebox(0,0)[lb]{\smash{slide}}}%
  \end{picture}%
\endgroup%

%% file: lemma_a.eps_tex
\begingroup%
  \makeatletter%
  \providecommand\color[2][]{%
    \errmessage{(Inkscape) Color is used for the text in Inkscape, but the package 'color.sty' is not loaded}%
    \renewcommand\color[2][]{}%
  }%
  \providecommand\transparent[1]{%
    \errmessage{(Inkscape) Transparency is used (non-zero) for the text in Inkscape, but the package 'transparent.sty' is not loaded}%
    \renewcommand\transparent[1]{}%
  }%
  \providecommand\rotatebox[2]{#2}%
  \ifx\svgwidth\undefined%
    \setlength{\unitlength}{234.84491703bp}%
    \ifx\svgscale\undefined%
      \relax%
    \else%
      \setlength{\unitlength}{\unitlength * \real{\svgscale}}%
    \fi%
  \else%
    \setlength{\unitlength}{\svgwidth}%
  \fi%
  \global\let\svgwidth\undefined%
  \global\let\svgscale\undefined%
  \makeatother%
  \begin{picture}(1,0.30018272)%
    \put(0,0){\includegraphics[width=\unitlength]{lemma_a.eps}}%
    \put(0.32443231,0.26634295){\color[rgb]{0,0,0}\makebox(0,0)[lb]{\smash{$\sb{-n-2}$}}}%
    \put(0.48179098,0.14716068){\color[rgb]{0,0,0}\makebox(0,0)[lb]{\smash{$2:1$}}}%
    \put(0.93596537,0.22682372){\color[rgb]{0,0,0}\makebox(0,0)[lb]{\smash{$\sb{-n-2}$}}}%
    \put(0.94800924,0.15456082){\color[rgb]{0,0,0}\makebox(0,0)[lb]{\smash{$\sb{-n-2}$}}}%
    \put(-0.00032934,0.27400021){\color[rgb]{0,0,0}\makebox(0,0)[lb]{\smash{$Y'=E_{-n}^*$}}}%
    \put(0.06170241,0.07715539){\color[rgb]{0,0,0}\makebox(0,0)[lb]{\smash{$\sb{0}$}}}%
    \put(0.65063329,0.07679777){\color[rgb]{0,0,0}\makebox(0,0)[lb]{\smash{$\sb{0}$}}}%
    \put(0.6409634,0.2711632){\color[rgb]{0,0,0}\makebox(0,0)[lb]{\smash{$X'$}}}%
  \end{picture}%
\endgroup%

%% file: lemma_b.eps_tex
\begingroup%
  \makeatletter%
  \providecommand\color[2][]{%
    \errmessage{(Inkscape) Color is used for the text in Inkscape, but the package 'color.sty' is not loaded}%
    \renewcommand\color[2][]{}%
  }%
  \providecommand\transparent[1]{%
    \errmessage{(Inkscape) Transparency is used (non-zero) for the text in Inkscape, but the package 'transparent.sty' is not loaded}%
    \renewcommand\transparent[1]{}%
  }%
  \providecommand\rotatebox[2]{#2}%
  \ifx\svgwidth\undefined%
    \setlength{\unitlength}{270.31957581bp}%
    \ifx\svgscale\undefined%
      \relax%
    \else%
      \setlength{\unitlength}{\unitlength * \real{\svgscale}}%
    \fi%
  \else%
    \setlength{\unitlength}{\svgwidth}%
  \fi%
  \global\let\svgwidth\undefined%
  \global\let\svgscale\undefined%
  \makeatother%
  \begin{picture}(1,0.57361683)%
    \put(0,0){\includegraphics[width=\unitlength]{lemma_b.eps}}%
    \put(0.31430645,0.00653162){\color[rgb]{0,0,0}\makebox(0,0)[lb]{\smash{$\sb{\text{full}~(-n-2)-\text{twist}}$}}}%
    \put(0.34780603,0.17631233){\color[rgb]{0,0,0}\makebox(0,0)[lb]{\smash{$\sb{-n-2}$}}}%
    \put(0.43219941,0.2427541){\color[rgb]{0,0,0}\makebox(0,0)[lb]{\smash{$\sb{-2n-2}$}}}%
    \put(0.19299884,0.10072727){\color[rgb]{0,0,0}\makebox(0,0)[lb]{\smash{$\sb{0}$}}}%
    \put(0.68001443,0.376187){\color[rgb]{0,0,0}\makebox(0,0)[lb]{\smash{$\sb{-n-2}$}}}%
    \put(0.94005793,0.54987732){\color[rgb]{0,0,0}\makebox(0,0)[lb]{\smash{$\sb{-2n-2}$}}}%
    \put(0.49712633,0.39648371){\color[rgb]{0,0,0}\makebox(0,0)[lb]{\smash{$\sb{0}$}}}%
    \put(0.83016435,0.05104955){\color[rgb]{0,0,0}\makebox(0,0)[lb]{\smash{$E(\xi_{-2n})$}}}%
    \put(0.8939461,0.23304619){\color[rgb]{0,0,0}\makebox(0,0)[lb]{\smash{$\sb{-2n}$}}}%
    \put(0.73817665,0.05320588){\color[rgb]{0,0,0}\makebox(0,0)[lb]{\smash{$\sb{0}$}}}%
    \put(0.7745013,0.23694741){\color[rgb]{0,0,0}\makebox(0,0)[lb]{\smash{$\sb{-n-2}$}}}%
    \put(0.276319,0.52193764){\color[rgb]{0,0,0}\makebox(0,0)[lb]{\smash{$\sb{-n-2}$}}}%
    \put(0.28678233,0.45915796){\color[rgb]{0,0,0}\makebox(0,0)[lb]{\smash{$\sb{-n-2}$}}}%
    \put(0.0284317,0.39159992){\color[rgb]{0,0,0}\makebox(0,0)[lb]{\smash{$\sb{0}$}}}%
    \put(0.02805283,0.5349417){\color[rgb]{0,0,0}\makebox(0,0)[lb]{\smash{$X'$}}}%
    \put(0.13137746,0.23248428){\color[rgb]{0,0,0}\makebox(0,0)[lb]{\smash{slide}}}%
    \put(0.47279039,0.2664372){\color[rgb]{0,0,0}\rotatebox{45}{\makebox(0,0)[lb]{\smash{isotopy}}}}%
    \put(0.87340323,0.29318374){\color[rgb]{0,0,0}\makebox(0,0)[lb]{\smash{isotopy}}}%
  \end{picture}%
\endgroup%

%% file: kkkkkk.eps_tex
\begingroup%
  \makeatletter%
  \providecommand\color[2][]{%
    \errmessage{(Inkscape) Color is used for the text in Inkscape, but the package 'color.sty' is not loaded}%
    \renewcommand\color[2][]{}%
  }%
  \providecommand\transparent[1]{%
    \errmessage{(Inkscape) Transparency is used (non-zero) for the text in Inkscape, but the package 'transparent.sty' is not loaded}%
    \renewcommand\transparent[1]{}%
  }%
  \providecommand\rotatebox[2]{#2}%
  \ifx\svgwidth\undefined%
    \setlength{\unitlength}{349.91515196bp}%
    \ifx\svgscale\undefined%
      \relax%
    \else%
      \setlength{\unitlength}{\unitlength * \real{\svgscale}}%
    \fi%
  \else%
    \setlength{\unitlength}{\svgwidth}%
  \fi%
  \global\let\svgwidth\undefined%
  \global\let\svgscale\undefined%
  \makeatother%
  \begin{picture}(1,0.29815645)%
    \put(0,0){\includegraphics[width=\unitlength]{kkkkkk.eps}}%
    \put(0.16731063,0.00812191){\color[rgb]{0,0,0}\makebox(0,0)[lb]{\smash{$S^1$}}}%
    \put(0.62576492,0.09667172){\color[rgb]{0,0,0}\makebox(0,0)[lb]{\smash{$S^1$}}}%
    \put(0.04524407,-0.03570277){\color[rgb]{0,0,0}\makebox(0,0)[lb]{\smash{$S^1\times D^3$}}}%
    \put(0.32060921,-0.03570277){\color[rgb]{0,0,0}\makebox(0,0)[lb]{\smash{$S^1\times D^3$}}}%
    \put(0.60266644,-0.03570277){\color[rgb]{0,0,0}\makebox(0,0)[lb]{\smash{$D^2\times S^2$}}}%
    \put(0.88572921,-0.03570277){\color[rgb]{0,0,0}\makebox(0,0)[lb]{\smash{$D^2\times S^2$}}}%
    \put(0.19623737,0.17786475){\color[rgb]{0,0,0}\makebox(0,0)[lb]{\smash{$2:1$}}}%
    \put(0.76017965,0.18174594){\color[rgb]{0,0,0}\makebox(0,0)[lb]{\smash{$2:1$}}}%
  \end{picture}%
\endgroup%

%% file: d_02.eps_tex
\begingroup%
  \makeatletter%
  \providecommand\color[2][]{%
    \errmessage{(Inkscape) Color is used for the text in Inkscape, but the package 'color.sty' is not loaded}%
    \renewcommand\color[2][]{}%
  }%
  \providecommand\transparent[1]{%
    \errmessage{(Inkscape) Transparency is used (non-zero) for the text in Inkscape, but the package 'transparent.sty' is not loaded}%
    \renewcommand\transparent[1]{}%
  }%
  \providecommand\rotatebox[2]{#2}%
  \ifx\svgwidth\undefined%
    \setlength{\unitlength}{358.56358409bp}%
    \ifx\svgscale\undefined%
      \relax%
    \else%
      \setlength{\unitlength}{\unitlength * \real{\svgscale}}%
    \fi%
  \else%
    \setlength{\unitlength}{\svgwidth}%
  \fi%
  \global\let\svgwidth\undefined%
  \global\let\svgscale\undefined%
  \makeatother%
  \begin{picture}(1,0.7124924)%
    \put(0,0){\includegraphics[width=\unitlength]{d_02.eps}}%
    \put(0.70237724,0.13021013){\color[rgb]{0,0,0}\makebox(0,0)[lb]{\smash{$\sb{1}$}}}%
    \put(0.83385412,0.04574619){\color[rgb]{0,0,0}\makebox(0,0)[lb]{\smash{$\sb{1}$}}}%
    \put(0.75926406,0.01492117){\color[rgb]{0,0,0}\makebox(0,0)[lb]{\smash{$\sb{n+2}$}}}%
    \put(0.80266157,0.17260332){\color[rgb]{0,0,0}\makebox(0,0)[lb]{\smash{$\sb{0}$}}}%
    \put(0.96594812,0.26731768){\color[rgb]{0,0,0}\makebox(0,0)[lb]{\smash{$\sb{-n-2}$}}}%
    \put(0.51136089,0.00364045){\color[rgb]{0,0,0}\makebox(0,0)[lb]{\smash{$\sb{0}$}}}%
    \put(0.63283939,0.00364045){\color[rgb]{0,0,0}\makebox(0,0)[lb]{\smash{$\sb{0}$}}}%
    \put(0.38543731,0.22017955){\color[rgb]{0,0,0}\makebox(0,0)[lb]{\smash{$\sb{2}$}}}%
    \put(0.4382405,0.22017955){\color[rgb]{0,0,0}\makebox(0,0)[lb]{\smash{$\sb{2}$}}}%
    \put(0.27983066,0.22017955){\color[rgb]{0,0,0}\makebox(0,0)[lb]{\smash{$\sb{2}$}}}%
    \put(0.33263398,0.22017955){\color[rgb]{0,0,0}\makebox(0,0)[lb]{\smash{$\sb{2}$}}}%
    \put(0.53614525,0.22017955){\color[rgb]{0,0,0}\makebox(0,0)[lb]{\smash{$\sb{2}$}}}%
    \put(0.58706021,0.22017955){\color[rgb]{0,0,0}\makebox(0,0)[lb]{\smash{$\sb{2}$}}}%
    \put(0.3252826,0.04615843){\color[rgb]{0,0,0}\makebox(0,0)[lb]{\smash{$\sb{2}$}}}%
    \put(0.63928848,0.39798929){\color[rgb]{0,0,0}\makebox(0,0)[lb]{\smash{$E(\xi_{-2n})$}}}%
    \put(0.68764466,0.53707735){\color[rgb]{0,0,0}\makebox(0,0)[lb]{\smash{$\sb{-2n}$}}}%
    \put(0.56781625,0.48240269){\color[rgb]{0,0,0}\makebox(0,0)[lb]{\smash{$\sb{0}$}}}%
    \put(0.65833169,0.61619172){\color[rgb]{0,0,0}\makebox(0,0)[lb]{\smash{$\sb{-n-2}$}}}%
    \put(0.80952043,0.55585867){\color[rgb]{0,0,0}\makebox(0,0)[lb]{\smash{$\sb{1}$}}}%
    \put(0.81175906,0.4463303){\color[rgb]{0,0,0}\makebox(0,0)[lb]{\smash{$\sb{1}$}}}%
    \put(0.86405745,0.53691095){\color[rgb]{0,0,0}\makebox(0,0)[lb]{\smash{$\sb{2(n+2)}$}}}%
    \put(0.64810726,0.12379207){\color[rgb]{0,0,0}\rotatebox{7.11043077}{\makebox(0,0)[lb]{\smash{$\approx$}}}}%
    \put(0.42042182,0.60272164){\color[rgb]{0,0,0}\makebox(0,0)[lb]{\smash{$\Phi_2$}}}%
    \put(0.7931968,0.32502695){\color[rgb]{0,0,0}\makebox(0,0)[lb]{\smash{$\Phi_1$}}}%
    \put(0.59787085,0.43713457){\color[rgb]{0,0,0}\makebox(0,0)[lb]{\smash{$h_2$}}}%
    \put(0.75636768,0.20214817){\color[rgb]{0,0,0}\makebox(0,0)[lb]{\smash{$h_1$}}}%
    \put(0.3634883,0.56805423){\color[rgb]{0,0,0}\makebox(0,0)[lb]{\smash{$\Z_{2n}$-covering}}}%
    \put(0.83947918,0.3517335){\color[rgb]{0,0,0}\makebox(0,0)[lb]{\smash{$\Z_2$-covering}}}%
    \put(0.41373413,0.39505343){\color[rgb]{0,0,0}\makebox(0,0)[lb]{\smash{$\sb{4n(n+2)}$}}}%
    \put(0.06293478,0.51755522){\color[rgb]{0,0,0}\makebox(0,0)[lb]{\smash{$\sb{0}$}}}%
    \put(0.14345022,0.64634425){\color[rgb]{0,0,0}\makebox(0,0)[lb]{\smash{$\sb{-n-2}$}}}%
    \put(0.23809809,0.66959394){\color[rgb]{0,0,0}\makebox(0,0)[lb]{\smash{$\sb{2n}$}}}%
    \put(0.21004591,0.42517916){\color[rgb]{0,0,0}\makebox(0,0)[lb]{\smash{$\sb{0}$}}}%
    \put(0.28656135,0.5529682){\color[rgb]{0,0,0}\makebox(0,0)[lb]{\smash{$\sb{-n-2}$}}}%
    \put(0.08328375,0.29383801){\color[rgb]{0,0,0}\makebox(0,0)[lb]{\smash{$\sb{2n}$}}}%
    \put(0.08945223,0.36276507){\color[rgb]{0,0,0}\makebox(0,0)[lb]{\smash{$\sb{-2n}$}}}%
    \put(0.24140506,0.29227149){\color[rgb]{0,0,0}\makebox(0,0)[lb]{\smash{$\sb{-2n}$}}}%
    \put(0.36786958,0.41994517){\color[rgb]{0,0,0}\makebox(0,0)[lb]{\smash{$\sb{1}$}}}%
    \put(0.37010821,0.31041679){\color[rgb]{0,0,0}\makebox(0,0)[lb]{\smash{$\sb{1}$}}}%
  \end{picture}%
\endgroup%

%% file: bump_fct.tex
\unitlength 0.1in
\begin{picture}( 27.0000, 17.7600)(  1.8500,-18.1800)
\put(13.2000,-1.1500){\makebox(0,0){$\rho_t$}}%
\put(27.0500,-15.5800){\makebox(0,0)[lb]{$t$}}%
\put(12.8000,-5.6800){\makebox(0,0)[rb]{$1$}}%
\put(13.0000,-15.6000){\makebox(0,0)[rt]{O}}%
%
{\color[named]{Black}{%
\special{pn 8}%
\special{pa 1320 1716}%
\special{pa 1320 216}%
\special{fp}%
\special{sh 1}%
\special{pa 1320 216}%
\special{pa 1300 282}%
\special{pa 1320 268}%
\special{pa 1340 282}%
\special{pa 1320 216}%
\special{fp}%
}}%
%
{\color[named]{Black}{%
\special{pn 8}%
\special{pa 1226 1524}%
\special{pa 2640 1524}%
\special{fp}%
\special{sh 1}%
\special{pa 2640 1524}%
\special{pa 2574 1504}%
\special{pa 2588 1524}%
\special{pa 2574 1544}%
\special{pa 2640 1524}%
\special{fp}%
}}%
{\color[named]{Black}{%
\special{pn 20}%
\special{pa 186 1488}%
\special{pa 190 1488}%
\special{pa 196 1490}%
\special{pa 216 1490}%
\special{pa 220 1492}%
\special{pa 246 1492}%
\special{pa 250 1494}%
\special{pa 270 1494}%
\special{pa 276 1496}%
\special{pa 300 1496}%
\special{pa 306 1498}%
\special{pa 330 1498}%
\special{pa 336 1500}%
\special{pa 366 1500}%
\special{pa 370 1502}%
\special{pa 400 1502}%
\special{pa 406 1504}%
\special{pa 436 1504}%
\special{pa 440 1506}%
\special{pa 470 1506}%
\special{pa 476 1508}%
\special{pa 516 1508}%
\special{pa 520 1510}%
\special{pa 560 1510}%
\special{pa 566 1512}%
\special{pa 616 1512}%
\special{pa 620 1514}%
\special{pa 680 1514}%
\special{pa 686 1516}%
\special{pa 780 1516}%
\special{pa 786 1518}%
\special{pa 1310 1518}%
\special{ip}%
\special{pa 1320 518}%
\special{pa 1530 518}%
\special{pa 1536 520}%
\special{pa 1560 520}%
\special{pa 1570 522}%
\special{pa 1576 522}%
\special{pa 1580 524}%
\special{pa 1586 526}%
\special{pa 1600 532}%
\special{pa 1610 538}%
\special{pa 1620 546}%
\special{pa 1626 550}%
\special{pa 1630 556}%
\special{pa 1636 564}%
\special{pa 1640 572}%
\special{pa 1646 580}%
\special{pa 1650 590}%
\special{pa 1656 602}%
\special{pa 1660 616}%
\special{pa 1666 630}%
\special{pa 1670 646}%
\special{pa 1676 664}%
\special{pa 1680 684}%
\special{pa 1686 706}%
\special{pa 1690 730}%
\special{pa 1696 756}%
\special{pa 1700 784}%
\special{pa 1706 814}%
\special{pa 1710 846}%
\special{pa 1716 878}%
\special{pa 1726 946}%
\special{pa 1746 1090}%
\special{pa 1756 1158}%
\special{pa 1760 1192}%
\special{pa 1766 1222}%
\special{pa 1770 1252}%
\special{pa 1776 1280}%
\special{pa 1780 1306}%
\special{pa 1786 1330}%
\special{pa 1790 1352}%
\special{pa 1796 1372}%
\special{pa 1800 1390}%
\special{pa 1806 1406}%
\special{pa 1810 1422}%
\special{pa 1816 1434}%
\special{pa 1820 1446}%
\special{pa 1826 1456}%
\special{pa 1830 1466}%
\special{pa 1836 1474}%
\special{pa 1840 1480}%
\special{pa 1846 1486}%
\special{pa 1850 1492}%
\special{pa 1860 1500}%
\special{pa 1870 1506}%
\special{pa 1886 1512}%
\special{pa 1890 1512}%
\special{pa 1896 1514}%
\special{pa 1900 1514}%
\special{pa 1910 1516}%
\special{pa 1916 1516}%
\special{pa 1920 1518}%
\special{pa 2150 1518}%
\special{fp}%
\special{pa 2160 518}%
\special{pa 2590 518}%
\special{pa 2596 520}%
\special{pa 2740 520}%
\special{pa 2746 522}%
\special{pa 2820 522}%
\special{pa 2826 524}%
\special{pa 2880 524}%
\special{pa 2886 526}%
\special{ip}%
}}%
\put(18.7000,-16.4500){\makebox(0,0){$\frac{1}{2}$}}%
%
{\color[named]{Black}{%
\special{pn 8}%
\special{pa 1360 518}%
\special{pa 2450 518}%
\special{dt 0.045}%
\special{pa 2450 518}%
\special{pa 2450 1518}%
\special{dt 0.045}%
}}%
%
{\color[named]{Black}{%
\special{pn 20}%
\special{pa 2076 1518}%
\special{pa 2450 1518}%
\special{fp}%
}}%
%
{\color[named]{Black}{%
\special{pn 8}%
\special{pa 1576 528}%
\special{pa 1576 1518}%
\special{dt 0.045}%
}}%
\put(15.9000,-16.2800){\makebox(0,0){$c$}}%
\put(12.8000,-5.6800){\makebox(0,0)[rb]{$1$}}%
\put(24.4500,-16.2800){\makebox(0,0){$1$}}%
\end{picture}%